\documentclass[12pt]{amsart}
\usepackage{amsmath}
\usepackage{mathrsfs}
\usepackage{psfrag}
\usepackage[cmtip,line,color,poly,all]{xy}
\usepackage{graphicx,amssymb,booktabs, verbatim, mathtools}
\usepackage[noadjust]{cite}
\usepackage{xspace}
\usepackage{blindtext}
\usepackage[pdfusetitle, pdfpagelabels, plainpages=false,
  bookmarks, bookmarksnumbered
]{hyperref} % links and references
\usepackage{bookmark}
\usepackage{subfigure}
\newif\iffloattoend

\iffloattoend
\usepackage[figuresonly, nomarkers, nolists, heads]{endfloat}

\fi
\newif\ifvc
\vcfalse
\ifvc
\input{vc}
\fi

\newcommand{\RR}{\mathbb R}%Reals
\newcommand{\ZZ}{\mathbb Z}%Integers
\newcommand{\QQ}{\mathbb Q}%Rationals
\newcommand{\CC}{\mathbb C}%Complex
\newcommand{\HH}{\mathbb H}%Quaternion
\newcommand{\OO}{\mathbb O}%Octonion
\renewcommand{\AA}{\mathbb A}%Albert algebra
\newcommand{\LL}{\mathbb S}%Light cone / spin factor
\newcommand{\JJ}{\mathbb J}%jordan algebra
\newcommand{\jordan}{\bullet}
\newcommand{\x}{\mathbf{x}}
\newcommand{\y}{\mathbf{y}}
\newcommand{\g}{\mathbf{g}}
\newcommand{\m}{\mathbf{m}}
\newcommand{\fn}{\mathrm{fn}}
\newcommand{\ip}[1]{\left \langle #1 \right \rangle}%Inner product
\newcommand{\ipA}[1]{\left \langle #1 \right \rangle_{\AA }}%Inner product in
\newcommand{\ipL}[1]{\left \langle #1 \right \rangle_{\LL }}%Inner product in
\newcommand{\fps}[1]{[\![#1]\!]}

\newcommand{\sE}{\mathscr{E}}

\DeclareMathOperator{\Aut}{Aut}

\DeclareMathOperator{\Tr}{Tr}%trace
\DeclareMathOperator{\Norm}{Norm}%Norm

\swapnumbers
\theoremstyle{plain}
\newtheorem{thm}[subsection]{Theorem}

\newtheorem{prop}[subsection]{Proposition}

\theoremstyle{definition}
\newtheorem{defn}[subsection]{Definition}
\newtheorem{remark}[subsection]{Remark}

\newtheorem{ex}[subsection]{Example}

\catcode`\@=11
\def\@secnumfont{\bfseries}
\catcode`\@12

\makeatletter
\begin{document}

\title[Exotic matrix models]{Exotic matrix models: the Albert algebra
and the spin factor}

\author{Paul E. Gunnells}
\address{Department of Mathematics and Statistics\\University of
Massachusetts\\Amherst, MA 01003-9305}
\email{gunnells@math.umass.edu}

\renewcommand{\setminus}{\smallsetminus}

\date{11 June 2017} 

\thanks{The author was partially supported by NSF grants DMS 1101640 and
1501832}



\begin{abstract}
The matrix models attached to real symmetric matrices and the
complex/quaternionic Hermitian matrices have been studied by many
authors.  These models correspond to three of the simple formally real
Jordan algebras over $\RR$.  Such algebras were classified by Jordan,
von Neumann, and Wigner in the 30s, and apart from these three there
are two others: (i) the spin factor $\LL=\LL_{1,n}$, an algebra built
on $\RR^{n+1}$, and (ii) the Albert algebra $\AA$ of $3\times 3$
Hermitian matrices over the octonions $\OO$.  In this paper we
investigate the matrix models attached to these remaining cases. 
\end{abstract}

\maketitle

\ifvc
\let\thefootnote\relax
\footnotetext{Base revision~\GITAbrHash, \GITAuthorDate,
\GITAuthorName.}
\fi

\section{Introduction}\label{s:intro}

\subsection{} Let $V = V_{\CC }$ be the real vector space of $n\times
n$ complex Hermitian matrices equipped with Lebesgue measure.  For
any polynomial function $f\colon V \rightarrow \RR $, define
\[
\ip{f}_{0} = \int_{V} f (X) \exp (-\Tr X^{2}/2)\,dX,
\]
where $\Tr (X) = \sum_{i} X_{ii}$ is the sum of diagonal entries, and
put 
\begin{equation}\label{eq:expectation}
\ip{f} = \ip{f}_{0}/\ip{1}_{0}.
\end{equation}
Let $k\geq 0$ be an integer, and consider \eqref{eq:expectation}
evaluated on the polynomial given by taking the trace of the $k$th power:
\begin{equation}\label{eq:basicint}
C_{\CC} (n,k) = \ip{\Tr X^{k}}.
\end{equation}
For $k$ odd \eqref{eq:basicint} clearly vanishes for all $n$.  On the
other hand, for $k$ even and $n$ fixed, it turns out that $C_{\CC}
(n,k)$ is an integer, and as a function of $n$ is a polynomial of
degree $(k+2)/2$ with integral coefficients.

Furthermore, the number $C_{\CC} (n,k)$ has the following remarkable
combinatorial interpretation.  Let $\Pi_{k}$ be a polygon with $k$
sides.  Any pairing $\pi$ of the sides of $\Pi_{k}$ determines a
topological surface $\Sigma (\pi)$ endowed with an embedded graph (the
images of the edges and vertices of $\Pi_{k}$).  Let $N (\pi)$ be the
number of vertices in this embedded graph.  Then we have
\begin{equation}\label{eq:complexsurfsum}
C_{\CC} (n,k) = \sum_{\pi} n^{N (\pi)},
\end{equation}
where the sum is taken over all oriented pairings of the edges of
$\Pi_{k}$ such that the resulting surface
$\Sigma_{\pi}$ is orientable.  For example, we have
\begin{equation*}
C_{\CC} (n,4) = 2n^{3}+n, \quad C_{\CC} (n,6) = 5n^{4}+10n^{2}, \quad C_{\CC}
(n,8) = 14n^{5}+70n^{3} + 21n. 
\end{equation*}
The pairings yielding $C_{\CC} (n,4)$ are shown in Figure \ref{fig:ccn4}.
For more information, see Harer--Zagier \cite{harer.zagier}, Etingof
\cite[\S 4]{etingof}, or Lando--Zvonkin \cite{lz}.

\begin{figure}[htb]
\psfrag{n3}{$n^{3}$}
\psfrag{n}{$n$}
\begin{center}
\includegraphics[scale=0.4]{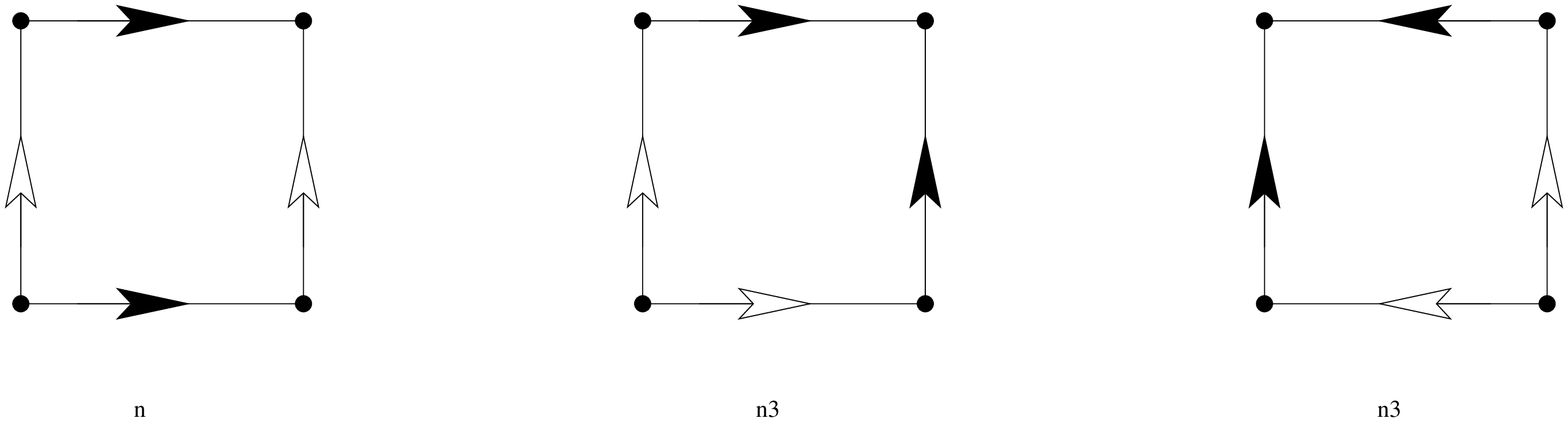}
\end{center}
\caption{Computing $C_\CC (n,4) = 2n^{3}+n$.}\label{fig:ccn4}
\end{figure}

\subsection{}
More generally, one can consider integrals over other spaces of
matrices, in particular, the space $V_{\RR}$ of $n\times n$ real
symmetric matrices, and the space $V_{\HH}$ of $n\times n$ quaternionic
Hermitian matrices.  The resulting matrix integrals were investigated
by Mulase--Waldron \cite{mulase.waldron}, who found explicit combinatorial
expressions for the analogues of \eqref{eq:basicint}.  In the real
symmetric case, they found 
\[
C_{\RR} (n,k) = 2^{-k/2}\sum_{\pi} n^{N (\pi)},
\]
where the sum is now taken over \emph{all} possible oriented pairings
of the edges of $\Pi_{k}$, regardless of whether the resulting surface
is orientable or not.  The quaternionic case is similar, except that
now \eqref{eq:complexsurfsum} takes the form
\[
C_{\HH} (n,k) = 2^{-k/2}\sum_{\pi} \alpha (\pi) n^{N(\pi)},
\]
where $\alpha (\pi) \in \{\pm 1 \}$ depends on the topology of $\Sigma_{\pi}$. 

\subsection{}\label{ss:assoctojord} The spaces of matrices $V_{\RR},
V_{\CC}, V_{\HH}$ have another interpretation that is less familiar:
they are examples of \emph{simple formally real Jordan algebras over $\RR$}
\cite{amrt, koecher, mccrimmon}.  Briefly, a \emph{Jordan algebra} over a field
$k$ is a nonassociative algebra over $k$ whose multiplication
satisfies $x\jordan y = y\jordan x$ and $ (x \jordan x) \jordan ( y
\jordan x ) = (( x \jordan x ) \jordan y) \jordan x $; it is
\emph{simple} if it cannot be written as a direct sum \dots Although
nonassociative, Jordan algebras are power-associative: if one puts
$x^{n} := x\jordan x^{n-1}$ for $n>1$, then $x^n$ can be computed as
$x\jordan \dotsb \jordan x$ with any choice of bracketing.

A Jordan algebra $A$ over $\RR$ is called \emph{formally real} if
$\sum_{i=1}^{n} x_{i}^{2} = 0$ implies each $x_{i}=0$.  It is known
that a real Jordan algebra being formally real is equivalent to it
having a \emph{positive definite trace form} $\Tr \colon A \rightarrow
\RR$.  This is a linear map satisfying $\Tr (x^{2}) > 0$ for all $x\in
A$, $x\not =0$, and one has in addition that the trace pairing $\Tr
(x\jordan y)$ is a positive definite quadratic form on $A$.

If the characteristic of $k$ is different from $2$, then any
associative algebra $A$ over $k$ can be turned into a Jordan algebra
by putting $x\jordan y = (xy+yx)/2$, where the multiplication on the
right is the usual multiplication in $A$. This is the Jordan structure
on the spaces $V_{\RR}$, $V_{\CC}$, $V_{\HH}$,
and the trace form $\Tr$ is of course the usual matrix trace. 

\subsection{}
Simple formally real Jordan algebras were classified by Jordan, von
Neumann, and Wigner in 1934 \cite{jnw}.  Apart from 
$V_{\RR}$, $V_\CC$, $V_{\HH}$, there are two others:
\begin{itemize}
\item The \emph{spin factor} $\LL = \LL_{1,n}$ of pairs $\x = (x_{0}, x) \in \RR \times
\RR^{n}$ equipped with the Jordan product $\x \jordan \y = (x_{0}y_{0}
+ x\cdot y, x_{0}y + y_{0}x)$, where $\cdot$ denotes the usual dot
product on $\RR^{n}$.
The trace form in this case is $\Tr (\x) = x_{0}$.
\item The \emph{Albert algebra} $\AA$ of $3\times 3$ Hermitian
matrices over the \emph{octonions} $\OO$, equipped with the same
Jordan product as $V_{\RR}$, $V_{\CC}$, $V_{\HH}$, and with the usual
trace as trace form.
\end{itemize}

\subsection{} Hence one has the natural problem of investigating the
``matrix models'' for the Jordan algebras $\LL$ and $\AA$, and of
understanding the underlying combinatorics.  In this paper we carry
this out.  In both cases we give a combinatorial method to compute the
expectations $\ip{\Tr X^{k}}_{\JJ}$, where $\JJ$ is one of the
algebras $\AA , \LL$, and where
\begin{equation}\label{eq:albertint}
\ip{\Tr X^{k}}_{\JJ} = \ip {\Tr X^{k}}_{0,\JJ}/\ip{1}_{0,\JJ }, \quad
\ip{f (X) }_{0,\JJ} = \int_{\JJ} f (X) \exp (-\Tr X^{2}/2)\,dX.
\end{equation}
The answers are quite different for the two algebras $\AA, \LL$,
although they do have some similarities with the classical cases.  For
the Albert algebra, the result (Theorem \ref{th:tracexk}) is given in
terms of contributions from (orientable and nonorientable) surfaces
glued together from polygons, as in the classical matrix algebra
cases.  For the spin factor, the result is given in terms of colored
one-manifolds glued together from closed intervals.

The next results describe how to compute the full perturbation series
attached to the trace monomials.  Let $t, g_{3}, g_{4}, g_{5}, \dotsc$
be formal parameters, and let $\QQ [g_{3},g_{4},\dotsc ]\fps{t}$ be
the ring of formal power series in $t$ whose coefficients are rational
polynomials in the $g_{k}$.  Then we compute
\begin{equation}\label{eq:perturb}
\ip{\exp \bigl(\sum_{k \geq 3}\Tr X^{k}g_{k}t^{k}\bigr)}_{\JJ} \in \QQ
[g_{3},g_{4},\dotsc ]\fps{t}
\end{equation}
in terms of assembling surfaces (respectively one-manifolds) from
polygons (resp.~intervals) of various sizes.  

The classical matrix algebras $V_{\RR}$, $V_{\CC}$, $V_\HH$ depend on
a parameter $n$.  This allows one to investigate the
expectations/perturbation series as a function of $n$.  This is also
true for the spin factor $\LL$, and accordingly we are able to
incorporate the parameter $n$ into our results.  The Albert algebra,
on the the other hand, is not part of an infinite family: there is no
general matrix model of $n\times n$ Hermitian matrices over $\OO$.
Indeed, it is exactly the failure of associativity for $\OO$ that
prevents such matrices for $n>3$ from having the structure of a Jordan
algebra.  For $n<3$, however, the matrix model does make sense: $n=1$
is just $V=\RR$, and $n=2$ is a special case of the spin factor $\LL$,
namely $\LL_{1,9}$.

Nevertheless, our results show how to express \eqref{eq:albertint} as
a \emph{polynomial} $C_{\OO} (n,k)$ in $n$ that for $n\leq 3$ agrees
with $\ip{\Tr X^{k}}$ evaluated over the appropriate space of
matrices.  Thus for $n\geq 4$ the combinatorial expansion allows us to
define the expectation $\ip{\Tr X^{k}}$, even though the algebraic
structure giving rise to it doesn't exist!  It would be interesting to
find an actual model computing these expectations for $n\geq 4$.

\subsection{} Here is a guide to the paper.  Part \ref{part:aa} treats
the Albert algebra $\AA $.  Section \ref{s:aa:background} gives the
basic definitions, including background on the octonions, and
discusses how the nonassociativity affects the trace computations.
Then in \S \ref{s:aa:traceint} we compute the expectation of the
traces of the powers in terms of gluings of polygons labelled by
octonions (Theorem \ref{th:tracexk}).  Our approach is a
generalization of that of Mulase--Waldron \cite{mulase.waldron},
although as one might expect, the nonassociativity of $\OO$ causes
some new wrinkles to appear.  The polynomials obtained by considering
the ``$n\times n$ Hermitian matrices over $\OO$'' are given in Table
\ref{tab:oct}. We end this part by explaining in \S \ref{s:aa:perturb}
how to compute the perturbation series (Theorem \ref{th:perturb}).
Next, part \ref{part:sf} gives a parallel treatment of the spin factor
$\LL$.  Background is recalled in \S \ref{s:sf:background}, the trace
integrals are computed in Theorem \ref{th:sf:expectation} in \S
\ref{s:sf:traceint}, and the perturbation series in Theorem
\ref{th:sf:perturb} in \S \ref{s:sf:perturb}.  One difference between
$\AA$ and $\LL$ is that for the latter we are able to give another
model that allows us to incorporate automorphisms, and to give a
simple generating function for the connected structures with their
automorphisms; this is done in \S \ref{s:sf:auts} in Theorem
\ref{th:secondmodel}.  

\subsection{Acknowledgments} We thank Ivan Mirkovic for the chance to
speak on the Harer--Zagier formula in his seminar, which sparked our
interest in matrix models and eventually led to this paper.  We thank
Daniel Briggs for helpful conversations.

\part{The Albert Algebra.}\label{part:aa}

\section{Background}\label{s:aa:background}

\subsection{} We begin with the octonions $\OO$.  As a real vector
space we have $\OO \simeq \RR^{8}$; for a basis we take elements $U =
\{e_{1},\dotsc ,e_{8}\}$ satisfying $e_{1}=1$ and $e_{i}^{2}=-1$ for
$i>1$.  For $i\not =j$ and $i,j>1$, we compute products $e_{i}e_{j}$
using the Fano mnemonic (Figure \ref{fig:fano}): $e_{i}e_{j}=\pm
e_{k}$ where $k$ is the third index on the line joining $i$ and $j$,
and where the sign is $+1$ (respectively, $-1$) if $i,j,k$ are in cyclic
order (respectively, out of cyclic order) with respect to the arrow on
the line through $i,j,k$.  For example, we have $e_{2}e_{3}=e_{4}$ and
$e_{2}e_{7}=-e_{8}$.  The unit $e_{1}$ is called \emph{real}, and
$e_{2},\dotsc ,e_{8}$ are called \emph{imaginary}.

We define the product $\alpha \beta$ of two general octonions $\alpha
= \sum_{i}a_{i}e_{i}$, $\beta = \sum_{i}b_{i}e_{i}$ using Figure
\ref{fig:fano} and linearity. The conjugate $\bar \alpha$ of $\alpha$
is defined by $\bar \alpha = a_{1}e_{1} - \sum_{i>1} a_{i}e_{i}$.  We
have two maps $\Tr , \Norm$ from $\OO$ to $\RR$ defined by
\[
\Tr \alpha = \alpha + \bar \alpha , \quad \Norm \alpha = \alpha\cdot 
\bar \alpha.
\]
These maps satisfy 
\[
\Tr (\alpha +\beta) = \Tr \alpha + \Tr \beta, \quad \Norm (\alpha
\beta) = \Norm (\alpha)\Norm (\beta).
\]

We remark that if $i,j,k$ lie on a line, then the $\RR$-subalgebra
spanned by $e_{1}$, $e_{i}$, $e_{j}$, $e_{k}$ is isomorphic to $\HH$.
In this case the triple product $e_{i}e_{j}e_{k}$ is associative.  If
$i,j,k$ do not lie on a line, then $e_{i}e_{j}e_{k}$ is not
associative, but it is \emph{alternative}: we have $(e_{i}e_{j})e_{k}
= - e_{i}(e_{j}e_{k})$.
\begin{figure}[htb]
\psfrag{2}{$\scriptstyle 2$}
\psfrag{3}{$\scriptstyle 3$}
\psfrag{4}{$\scriptstyle 4$}
\psfrag{5}{$\scriptstyle 5$}
\psfrag{6}{$\scriptstyle 6$}
\psfrag{7}{$\scriptstyle 7$}
\psfrag{8}{$\scriptstyle 8$}
\begin{center}
\includegraphics[scale=0.3]{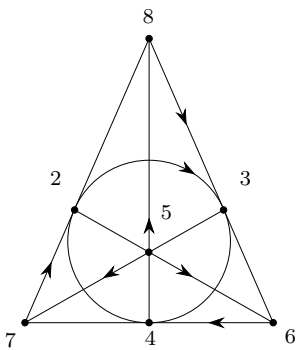}
\end{center}
\caption{The Fano plane mnemonic for multiplication in $\OO$.\label{fig:fano}}
\end{figure}

\subsection{} Now we define the \emph{Albert algebra} $\AA$.  As a
real vector space $\AA$ is the space of $3\times 3$ Hermitian matrices
over $\OO$:
\[
\AA = \Biggl\{\biggl(\begin{array}{ccc}
a_{1}&\alpha_{1}&\alpha_{2}\\
\bar \alpha_{1}&a_{2}&\alpha_{3}\\
\bar \alpha_{2}&\bar \alpha_{3}&a_{3}
\end{array}\biggr)\Biggm|a_{i} \in \RR , \alpha_{i}\in \OO\Biggr\}.
\]

As described above, $\AA$ becomes a formally real simple
Jordan algebra after we put 
\begin{equation}\label{eq:jord}
X\jordan Y = (XY+YX)/2,
\end{equation}
where the product on the right is the usual matrix product.  
We define powers $X^{k}$ by 
\begin{equation}\label{eq:powers}
X^{k} = \begin{cases}
X&\text{if $k=1$,}\\
X\jordan X^{k-1}&\text{if $k>1$}.
\end{cases}
\end{equation}

The powers $X^{k}$ are the first place where the nonassociativity of
$\OO$ has an effect.  As mentioned before, the spaces of Hermitian
matrices over $\RR$, $\CC$, $\HH$ can be turned into Jordan algebras
using \eqref{eq:jord}.  For them, the power $X^{k}$ is exactly $X
\dotsb X$ ($k$ factors), where the implied product is the usual
associative matrix product.  But since $\OO$ is not associative, the
expression $X \dotsb X$ is not well-defined.  Indeed, this expression
must be carefully bracketed using \eqref{eq:jord}  and \eqref{eq:powers} 
to guarantee that $\AA$ is power-associative (as all
Jordan algebras are).  Proposition \ref{prop:powers} below explains
how to evaluate $X^{k}$ using the ordinary matrix product.  Before we
can state it we need more notation.

\subsection{} It is well known that any bracketing of a word of length
$k$ in a nonassociative algebra can be encoded by a rooted binary tree
with $k$ leaves, and that both bracketings and such trees are counted
by the Catalan numbers.  In our case, not all bracketings arise when
computing $X^{k}$.  Let us call a rooted binary tree \emph{fully
nested} if one and only one vertex has two leaves.  We also call a
bracketing fully nested if the corresponding tree is.  In Figure
\ref{fig:bracketings}, for example, we see the five bracketings of
$XXXX$ with their corresponding rooted binary trees.  Only the first
four are fully nested.

\begin{prop}\label{prop:powers}
\leavevmode
\begin{enumerate}
\item Let $k\geq 3$.  There are $2^{k-2}$ fully nested bracketings of
a word of length $k$ in a nonassociative algebra.
\item In $\AA$, we have 
\[
X^{k} = \frac{1}{2^{k-2}}\sum_{P} P (X\cdots X),
\]
where the sum is taken over all fully nested bracketings $P$ of $X 
\dotsb X$ ($k$ factors).
\end{enumerate}
\end{prop}

\begin{proof}
Any fully nested rooted binary tree can be encoded by a word of length
$k-2$ in the symbols $L,R$: one reads down from the root and writes
down the sequence of children that are non-leaves.  This proves (i).
The proof of (ii) is a straightforward application of the definition
\eqref{eq:powers} of $X^{k}$ and the Jordan product \eqref{eq:jord}.
\end{proof}

For later use, we introduce some notation.  Given any word $w$ of
length $k$ representing a product of elements in an algebra, we define
\[
[w]_{\fn} = \frac{1}{2^{k-2}} \sum_{P} P (w),
\]
where the sum is taken over all fully nested bracketings of $w$.

\begin{figure}[htb]
\psfrag{x1}{$((XX)X)X$}
\psfrag{x2}{$(X (XX))X$}
\psfrag{x3}{$X ((XX)X)$}
\psfrag{x4}{$X (X (XX))$}
\psfrag{x5}{$(XX) (XX)$}
\psfrag{ll}{$LL$}
\psfrag{rl}{$RL$}
\psfrag{lr}{$LR$}
\psfrag{rr}{$RR$}
\begin{center}
\includegraphics[scale=0.5]{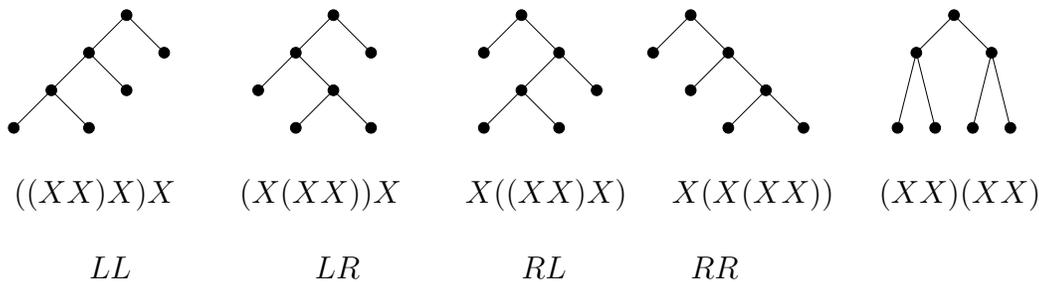}
\end{center}
\caption{The five bracketings of $XXXX$ and the corresponding rooted
binary trees.  The first four expressions are fully nested, and the
encodings of their trees in terms of words over $\{L,R \}$ are shown
below.\label{fig:bracketings}}
\end{figure}

\subsection{}
We conclude by stating Wick's theorem, which is the fundamental
combinatorial result used in evaluating Gaussian integrals.  Let $f
(x)$ be a polynomial function on $\RR$ and let $dx$ be the usual
Lebesgue measure on $\RR $.  Define
\begin{equation}\label{eq:realpairing}
\ip{f}_{0,\RR } = \int_{\RR} f (x) \exp (-x^{2}/2)\,dx, \quad
\ip{f}_{\RR} =  \ip{f}_{0,\RR}/\ip{1}_{0,\RR }.
\end{equation}
Recall that a \emph{pairing} on a finite set $S=\{1,\dotsc ,n \}$ is a
partition of $S$ into nonintersecting subsets of order $2$.  If $n$ is
even, there are $w (n) = (n-1)!! = (n-1) (n-3)\dotsb 1$ pairings of $S$.

\begin{thm}\label{th:wick}
If $k$ is odd, then $\ip{x^{k}}_{\RR }$ vanishes.  If $k$ is even,
then $\ip{x^{k}}_{\RR } = w (k)$, the number of pairings of a set of
order $k$.
\end{thm}

The proof of Theorem \ref{th:wick} can be found in many places, for
example \cite[\S 3.1]{etingof} and \cite{lz}.  One learns from the
proof that the value $(k-1)!!$ actually arises combinatorially: it
should be thought of as counting all pairings of the different $x$'s
in the expression $x\dotsb x$ ($k$ factors).  Thus, for example, we
have $\ip{x^{4}}_{\RR}=3$.  If we use subscripts to distinguish the
positions and write $xxxx$ as $x_{1}x_{2}x_{3}x_{4}$, then the $3$
arises from the three pairings
\[
(x_{1}x_{2}) (x_{3}x_{4}),\quad (x_{1}x_{3}) (x_{2}x_{4}), \quad (x_{1}x_{4}) (x_{2}x_{3}).
\]

The numbers $w (k)$ are called the \emph{Wick numbers}.

\section{Computation of the basic trace integral}\label{s:aa:traceint}

\subsection{} We begin by recalling some notation from Mulase--Waldron
\cite{mulase.waldron}, adapted to our case.  Any variable $n\times n$
matrix $X$ over $\OO$ can be written as
\begin{equation}\label{eq:Aexp}
X = \sum_{i=1}^{8} A^{i}e_{i},
\end{equation}
where the $A^{i}$ are $n\times n$ matrices of real variables.
Furthermore, $X$ is Hermitian if and only if $A^{1}$ is symmetric
($A^{1} = (A^{1})^{t}$) and $A^{i}$ is antisymmetric for $i>1$ ($A^{i}
=- (A^{i})^{t}$).  

To evaluate $\ipA{\Tr X^{k}}$, we must first compute the real
polynomial $\Tr X^{k}$.  This is accomplished by Proposition
\ref{prop:trpol} below.  Before stating the result, let $R_{k} \subset
\{1,\dotsc ,8 \}^{k}$ be the set of all tuples $(i_{1},\dotsc ,i_{k})$
such that the product of octonion units $e_{i_{1}}\dotsb e_{i_{k}}$
with respect to some fixed bracketing is \emph{real}.  We remark that
this property is independent of whatever bracketing of this product is
used to evaluate it, although the actual value, which must be $\pm 1$,
depends on the bracketing.  To see this, consider any product of
octonion units $e_{i_{1}}\dotsb e_{i_{k}}$, whether real-valued or
not.  After choosing a bracketing and evaluating, one obtains $\pm
e_{i}$ for some $i=1,\dotsc ,8$.  Any two bracketings are related
through a sequence of replacements of the form $e_{a}
(e_{b}e_{c})\mapsto (e_{a}e_{b})e_{c}$ applied to three consecutive
factors.  If the indices $a,b,c$ lie on a line or at least one equals
$1$, then $e_{a} (e_{b}e_{c})= (e_{a}e_{b})e_{c}$.  Otherwise,
alternativity implies $e_{a} (e_{b}e_{c})= -(e_{a}e_{b})e_{c}$.  This
implies that the unbracketed product $e_{i_{1}}\dotsb e_{i_{k}}$ is
well-defined up to sign, and in particular requiring it be real-valued
makes sense.

\begin{prop}\label{prop:trpol}
Let $n\leq 3$ and let $X$ be as in \eqref{eq:Aexp}.  Then we have 
\begin{equation}\label{eq:trpol}
\Tr X^{k} = \sum_{\substack{1\leq j_{1},\dotsc ,j_{k}\leq n,\\
(i_{1},\dotsc ,i_{k})\in R_{k}}}
A_{j_{1}j_{2}}^{i_{1}}A_{j_{2}j_{3}}^{i_{2}}\dotsb
A_{j_{k}j_{1}}^{i_{k}} [e_{i_{1}}\dotsb e_{i_{k}}]_{\fn}.
\end{equation}
\end{prop}

\begin{proof}
Let $M$ be any $n\times n$ matrix over an associative algebra.  Then
it is well known that
\begin{equation}\label{eq:trm}
\Tr M^{k} = \sum_{1\leq j_{1},\dotsc ,j_{k}\leq n}
M_{j_{1}j_{2}}M_{j_{2}j_{3}}\dotsb M_{j_{k}j_{1}}.
\end{equation}
For the convenience of the reader, we recall the proof. Let $G$ be the
directed graph on $n$ vertices with directed edges from each vertex to
another and with a loop on each vertex.  Then $M$ can be interpreted
as the weighted adjacency matrix of $G$, where the edge from vertex
$i$ to $j$ is labelled with $M_{ij}$.  The entries of $M^{k}$
correspond to $k$-step walks on $G$: $(M^{k})_{ij}$ is the sum over
all $k$-step walks from $i$ to $j$ of the product of the weights along
each walk.  The result \eqref{eq:trm} follows from the observation
that a walk contributes to $\Tr M^{k}$ if and only if it starts and
stops at the same vertex.

Now we apply \eqref{eq:trm} to $X=\sum A^{i}e_{i}$.  If we expand the
expression for $\Tr X^{k}$ using Proposition \ref{prop:powers}, we
obtain
\begin{equation}\label{eq:trx}
\Tr X^{k} = \sum_{\substack{1\leq j_{1},\dotsc ,j_{k}\leq n}}
A_{j_{1}j_{2}}^{i_{1}}A_{j_{2}j_{3}}^{i_{2}}\dotsb
A_{j_{k}j_{1}}^{i_{k}} [e_{i_{1}}\dotsb e_{i_{k}}]_{\fn}.
\end{equation}
We claim in \eqref{eq:trx} we only need to consider tuples
$(i_{1},\dotsc ,i_{k})$ that are in $R_{k}$.  Certainly if
$(i_{1},\dotsc ,i_{k})\in R_{k}$, then $[e_{i_{1}}\dotsb
e_{i_{k}}]_{\fn}$ is real.  On the other hand, if $(i_{1},\dotsc
,i_{k})\not \in R_{k}$, then it can happen that $[e_{i_{1}}\dotsb
e_{i_{k}}]_{\fn}$ is real even if the individual terms in the
conmputation of $[\phantom{a}]_{\fn}$ are not real.  However, the
discussion immediately before the statement of Proposition
\ref{prop:trpol} shows that this is possible only if $[e_{i_{1}}\dotsb
e_{i_{k}}]_{\fn}$ vanishes.  This completes the proof.
\end{proof}

\subsection{} Let $\Pi_{k}$ be a polygon with $k$ sides.  Suppose for
now that $k$ is even.  We fix an embedding of $\Pi_{k}$ in the plane
and distinguish one vertex with a star.  We write $\Pi^{*}_{k}$ to
indicate that this has been done.

Let $\sE = \sE (\Pi_{k}) = \{E_{1},\dotsc ,E_{k} \}$ be the edges of
$\Pi^{*}_{k}$, numbered counterclockwise around $\Pi^{*}_{k}$ with the
first and last edges adjacent to the distinguished vertex.  By an
\emph{oriented gluing} $\pi$ of $\Pi^{*}_{k}$ we mean a choice of
orientation for each $E\in \sE$ together with a pairing on $\sE$.  An
oriented gluing determines a topological surface $\Sigma (\pi)$ by
gluing each edge to its pair consistent with the orientations.  We say
that two edges $E,E'$ are glued \emph{without a twist} if their
orientations are \emph{opposite} as we move from the first to the
second around the boundary of $\Pi^{*}_{k}$.  Otherwise we say they
are glued \emph{with a twist}. For example, in Figure \ref{fig:ccn4}
all pairs of edges are glued without a twist, and in Figure
\ref{fig:bigonexample} the gluing on the right is done with a twist.
Let $N (\pi)$ be the number of equivalence classes of the vertices
under the gluing.

Suppose $\Pi^{*}_{k}$ has been equipped with an oriented gluing $\pi$.  We
say that a function $f \colon \sE \rightarrow U$ from the edges to the
units $\{e_{1},\dotsc ,e_{8} \}$ is \emph{compatible} with $\pi$ if
the following hold:
\begin{enumerate}
\item The product $\prod_{E\in \sE } f (E)$ is real-valued.
\item If $E, E'$ are identified by $\pi$, then $f (E)
= f (E')$.
\end{enumerate}

\begin{defn}\label{def:evalgluing}
Let $\pi$ be an oriented gluing of $\Pi^{*}_{k}$, with $k$ even, and
let $f\colon \sE \rightarrow U$ be compatible with $\pi$.  Then we define
the \emph{value} $\Omega (\pi , f) $ of the pair $(\pi ,f)$ to be 
\begin{equation}\label{eq:omega}
\Omega (\pi , f) = \alpha (\pi ,f) [f (E_{1})\dotsb f (E_{k})]_\fn  \in \QQ,
\end{equation}
where the sign $\alpha (\pi , f) \in \{\pm 1 \}$ is defined by the
following rules:
\begin{enumerate}
\item $\alpha (\pi ,f)$ is computed by taking a product of signs
$\alpha (E,E')$ over all edge
pairs $E$, $E' = \pi (E)$.
\item If $f (E) = f (E')$ is $e_{1}$, then $\alpha (E,E') = 1$.
\item If $f (E) = f (E')$ is imaginary, then $\alpha (E,E') = 1$ if
$E$ is glued to $E'$ \emph{without} a twist.
\item If $f (E) = f (E')$ is imaginary, then $\alpha (E,E') = -1$ if
$E$ is glued to $E'$ \emph{with} a twist.
\end{enumerate}

\end{defn}

\begin{ex}\label{ex:hexagon.one}
% [4, 5, 4, 5, 6, 6] -5/8
We give an example of computing $\Omega (\pi ,f)$ for $k=6$.  Suppose
$\pi$ is as in Figure \ref{fig:hexagon}.  Suppose that $f$ satisfies
$f (E_{1}) = f (E_{3}) = e_{4}$, $f (E_{2})=f (E_{4}) = e_{5}$, and $f
(E_{5})=f (E_{6}) = e_{6}$.  We have $\alpha (E_{1},E_{3}) = \alpha
(E_{5},E_{6})=-1$ and $\alpha (E_{2},E_{4})=1$.  There are $16$ fully
nested bracketings for the expression
$e_{4}e_{5}e_{4}e_{5}e_{6}e_{6}$.  The contribution $\Omega (\pi ,f)$
is $-5/8$.  This example also shows that $\Omega (\pi ,f)$ can be
nonintegral, thanks to the averaging in \eqref{eq:omega}.
\end{ex}

\subsection{}
We are now ready to state our first main result:
\begin{thm}\label{th:tracexk}
Let $k\geq 2$ be even and let $\Tr \colon \AA \rightarrow \RR$ be the
trace.  Let $X$ be a $3\times 3$ Hermitian matrix of variables as in
\eqref{eq:Aexp}. Then we have
\begin{equation}\label{eq:th1}
\ipA{\Tr X^{k}} = 2^{-k/2}\sum_{\pi} \sum_{f} \Omega (\pi ,f) 3^{N (\pi)}, 
\end{equation}
where the first sum is taken over all oriented gluings  of the edges of
$\Pi_{k}^{*}$, and the second sum is taken over all functions $f\colon \sE
\rightarrow U$ compatible with $\pi $.
\end{thm}

\begin{proof}
We use the expression \eqref{eq:trpol} for $\ipA{\Tr X^{k}}$ together
with Wick's theorem (Theorem \ref{th:wick}).  A term in
\eqref{eq:trpol} can contribute to $\ipA{\Tr X^{k}}$ if and only if
there is a complete pairing of the variables $A^{i}_{jj'}$, and if the
product $e_{i_{1}}\dotsb e_{i_{k}}$ is real.  Following
\cite{harer.zagier,mulase.waldron}, any such pairing can be visualized
by labelling the polygon $\Pi_k^{*}$.  We assign the vertices the
labels $j_{1},\dotsc ,j_{k}$, starting with the distinguished vertex
and proceeding clockwise.  The edge between $j_{r}$ and $j_{r+1}$ then
corresponds to the variable $A^{i_{r}}_{j_{r}j_{r+1}}$.  If two
variables are to be paired in Theorem
\ref{th:wick}, we can represent this by gluing the corresponding edges
together in $\Pi_{k}^{*}$.    

Thus fix a monomial in \eqref{eq:trpol}, and let $A^{i}_{ab}$,
$A^{i'}_{cd}$ be two of the variables that are to be paired (Figure
\ref{fig:twoedges}).  For these variables to be equal, we must have
$i=i'$, along with some other constraints.

First, if $i=i'=1$, then $A^{1}$ is symmetric.  The variables can be
equal if either (i) $a=c$ and $b=d$ or (ii) $a=d$ and $b=c$.  Case (i)
corresponds to gluing with a twist, whereas (ii) corresponds to gluing
without a twist.  In both cases, the symmetry of $A^{1}$ does not
affect the sign of this monomial.  

On the other hand, if $i=i'>1$, then $A^{i}$ is antisymmetric.  Again
the variables can be equal if either (i) $a=c$ and $b=d$ or (ii) $a=d$
and $b=c$, and again these represent gluing with and without a twist
respectively.  As before case (ii) introduces no extra sign, but case
(i) does.  For these pairings, we have the same imaginary unit
attached to each edge.  These units square to $-1$, which have the
effect of flipping the signs in (i), (ii). Thus these two possiblities
exactly correspond to the computation of $\alpha (\pi ,f)$ in
Definition \ref{def:evalgluing}.

After a pairing $\pi$ has been chosen on the perimeter of
$\Pi_{k}^{*}$, we must compute how many different monomials the
pairing can contribute to.  This is done by (i) summing over all
assignments of units to edges compatible with $\pi$, and (ii) varying
the subscripts $j_{1},\dotsc ,j_{k}$ over all possibilities.  The
first is handled by summing over all $f$ compatible with $\pi$.  For
the second, after applying $\pi$ one finds that certain subscripts
$j_{l}$ must be equal, but apart from that one can allow them to range
over any of the three possibilities $1,2,3$ (since $\AA$ consists of $3\times
3$ matrices over $\OO$).  The number of different classes of
subscripts is the same as $N (\pi )$, which is the factor appearing in
\eqref{eq:th1}.

Finally we must account for the factor $2^{-k/2}$.  This arises
because in our correspondence between gluings of $\Pi_{k}^{*}$ and
monomials giving a nontrivial pairing we are overcounting.  Indeed,
the variables $A^{i}_{ab}$ and $A^{i'}_{cd}$ identified by the gluing
$\pi$ are the same, regardless of whether we glue with a twist or not.
This means we see each pair of variables twice, once for gluing with a
twist and once for gluing without.  Since there are $k/2$ pairs of
edges to be glued, we must divide by $2^{k/2}$.  This completes the
proof.
\end{proof}

\begin{figure}[htb]
\psfrag{a1}{$A^{i}_{ab}$}
\psfrag{a2}{$A^{i'}_{cd}$}
\begin{center}
\includegraphics[scale=0.4]{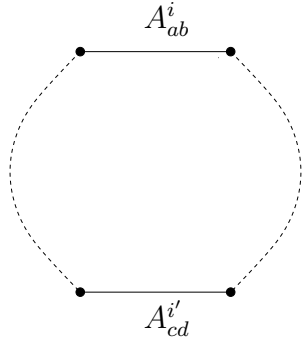}
\end{center}
\caption{Two edges in $\Pi_{k}^{*}$\label{fig:twoedges}}
\end{figure}
\begin{ex}\label{ex:bigon}
Consider computing $\ipA{\Tr X^{2}}$.  There are two gluings $\pi_{1},
\pi_{2}$ of the bigon $\Pi_{2}^{*}$, shown in Figure
\ref{fig:bigonexample}.  We have $N (\pi_{1}) = 2$ and $N (\pi_{2}) =
1$.  We have $8$ possibilities for labelling the edges $E,E'$ in each.
In $\pi_{1}$, the edges are glued without a twist, so all signs
$\alpha (E,E')$ are positive.  Thus $\Omega (\pi _{1}, f) = 1$ for any
of the $8$ possible $f$, and we obtain $8\cdot 3^{2} = 72$ for this
gluing.  In $\pi_{2}$, the edges are glued with a twist.  This means
that when $f (E) = e_{1}$ is real, we have $\alpha (E,E') =1$, and
when $f (E)$ is imaginary, we have $\alpha (E,E') = -1$.  Summing over
all $f$ gives $-6$, and we get a contribution of $-6\cdot 3 = -18$
from this gluing.  The final result is $\ipA{\Tr X^{2}} = \frac{1}{2}
(8\cdot 3^{2}-6\cdot 3) = 27$.  This can be checked directly.  We have
\begin{equation}\label{eq:tr2example}
\Tr X^{2} = \sum_{j=1}^{3}(A_{jj}^{1})^{2} + \sum_{\substack{1\leq i\leq 8\\
1\leq j<j'\leq 3}} 2(A_{jj'}^{i})^{2}.
\end{equation}
Applying Theorem \ref{th:wick}, we find $\ipA{(A_{jj}^{1})^{2}} = 1$ and
$\ipA{(A_{jj'}^{i})^{2}} = 1/2$, which with \eqref{eq:tr2example} yields
$27$.
\end{ex}

\begin{ex}\label{ex:square}
Figure \ref{fig:squareexample} shows the $12$ oriented gluings of
$\Pi_{4}^{*}$ with their contributions.  The result is $2^{-2}
(128\cdot 3^3 - 240\cdot 3^2 + 124\cdot 3) = 417$.  Note that in this
example, as in Example \ref{ex:bigon}, the products of the $e_{i}$ are
all associative.  This follows since there are only at most two
different units appearing in any product, so each expression is being
computed in a subalgebra isomorphic to $\HH$.  Thus there is no need
to compute the fully nested bracketings.
\end{ex}

\begin{ex}\label{ex:hexagonexample}
The evaluation of $\ipA{\Tr X^{6}}$ uses gluings of the hexagon
$\Pi_{6}^{*}$.  There are $15$ possible ways to pair the edges of
$\Pi_{6}^{*}$, and each pairing has $8$ different twisting patterns.
For each pairing with twists there are $512$ possible assignment of
units to the edges, and for each assignment there are $16$ fully
nested bracketings to compute.  After evaluating $983040$ terms the
final answer is $\ipA{\Tr X^{6}} = 7533$.
\end{ex}

\begin{figure}[htb]
\psfrag{st}{$*$}
\psfrag{n3}{$3^{3}$}
\psfrag{n2}{$8\cdot 3^{2}$}
\psfrag{n}{$-6\cdot 3$}
\begin{center}
\includegraphics[scale=0.4]{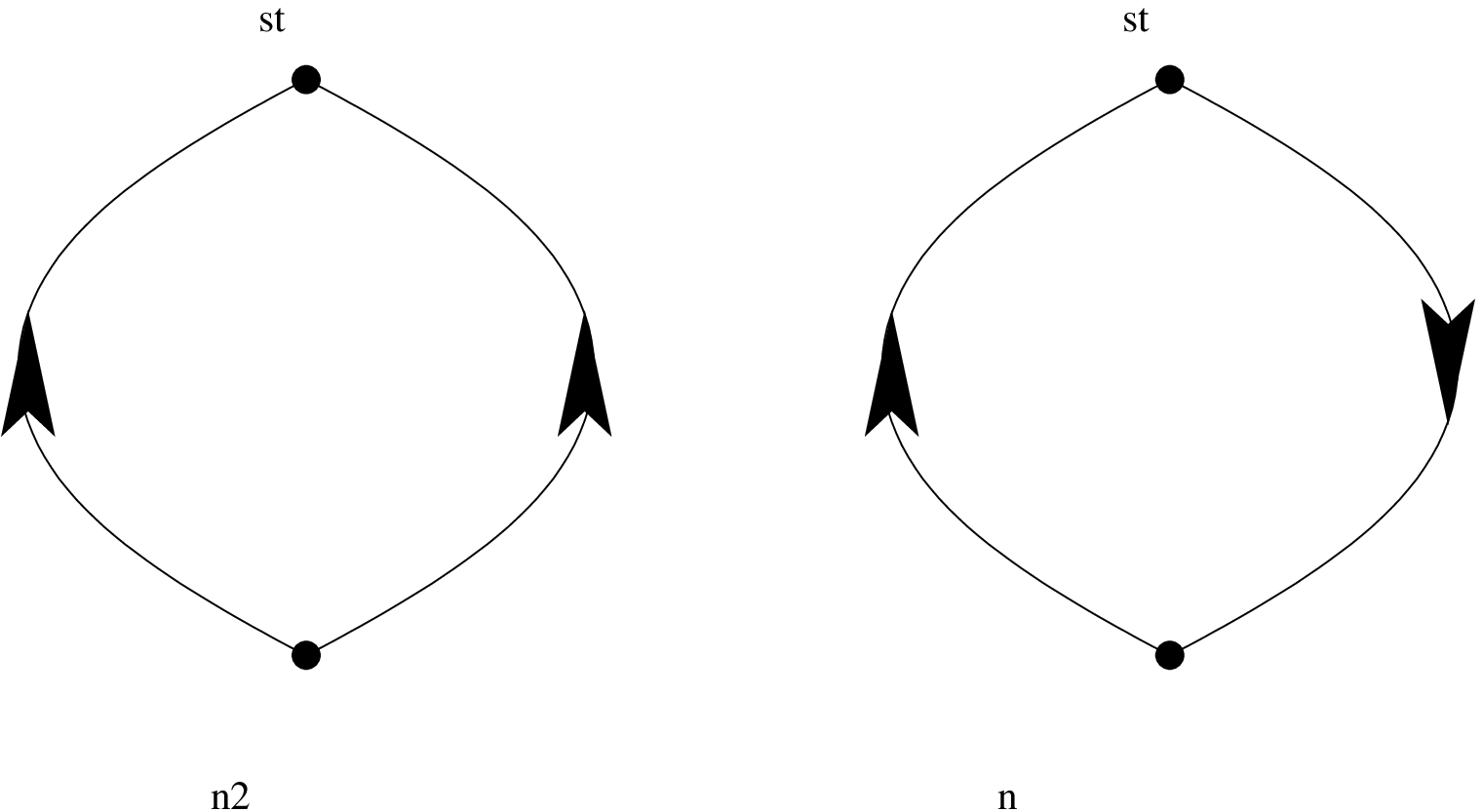}
\end{center}
\caption{The two gluings $\pi_{1}$, $\pi_{2}$ used in computing
$\ipA{\Tr X^{2}}$ on $\AA $.}\label{fig:bigonexample}
\end{figure}

\begin{figure}[htb]
\psfrag{st}{$*$}
\psfrag{t1}{$-20\cdot 3$}
\psfrag{t2}{$36\cdot 3$}
\psfrag{t3}{$36\cdot 3$}
\psfrag{t4}{$-48\cdot 3^{2}$}
\psfrag{a1}{$64\cdot 3^{3}$}
\psfrag{a2}{$-48\cdot 3^{2}$}
\psfrag{a3}{$-48\cdot 3^{2}$}
\psfrag{a4}{$36\cdot 3$}
\psfrag{b1}{$64\cdot 3^{3}$}
\psfrag{b2}{$-48\cdot 3^{2}$}
\psfrag{b3}{$-48\cdot 3^{2}$}
\psfrag{v4}{$36\cdot 3$}
\begin{center}
\includegraphics[scale=0.4]{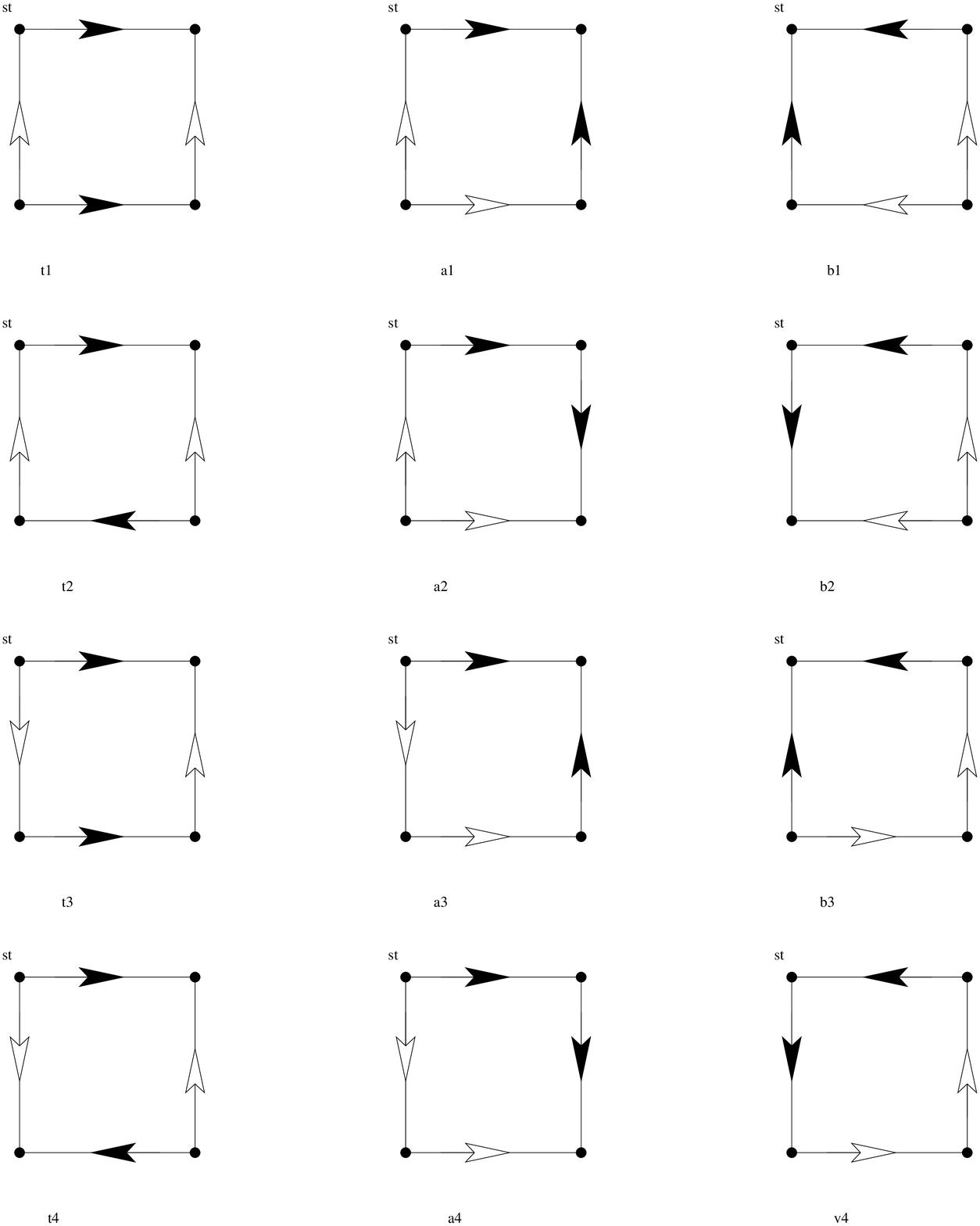}
\end{center}
\caption{Computing $\ipA{\Tr X^{4}}$ on $\AA $.}\label{fig:squareexample}
\end{figure}

\begin{figure}[htb]
\psfrag{st}{$*$}
\psfrag{e1}{$E_{1}$}
\psfrag{e2}{$E_{2}$}
\psfrag{e3}{$E_{3}$}
\psfrag{e4}{$E_{4}$}
\psfrag{e5}{$E_{5}$}
\psfrag{e6}{$E_{6}$}
\begin{center}
\includegraphics[scale=0.4]{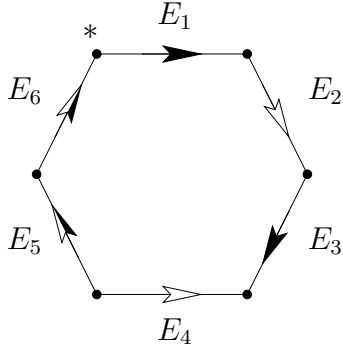}
\end{center}
\caption{One term in evaluating $\ipA{\Tr X^{6}}$ on $\AA $. The pairs
$\{E_{1},E_{3} \}, \{E_{5},E_{6} \}$ have been glued with a twist, the
pair $\{E_{2}, E_{4} \}$ without.  We have $N (\pi) = 1$.  This term
contributes $-153\cdot 3=-459.$ }\label{fig:hexagon}
\end{figure}

\begin{remark}\label{rem:previous}
In the computation of $\ipA{\Tr X^{k}}$, one can replace $3^{N (\pi)}$
in \eqref{eq:th1} with $2^{N (\pi)}$.  One then finds the result of
evaluating $\ipA{\Tr X^{k}}$ on the $2\times 2$ Hermitian matrices over
$\OO$.  In fact, one can replace $3^{N (\pi)}$ with $n^{N (\pi)}$ for
an indeterminate $n$, and one obtains a rational polynomial $ C_{\OO}
(n,k)$.  By analogy with the associative cases ($\RR$, $\CC$, $\HH$),
one can regard this as computing the expectation of $\Tr X^{k}$ on an
``algebra'' of $n\times n$ Hermitian matrices over $\OO$.  Of course
this is only an analogy: there is no such algebra, and the computation
is purely formal.

Some examples of the polynomials $C_{\OO} (n,k)$ are given in Table
\ref{tab:oct}.  One sees from the table that these polynomials
apparently have surprising properties.  For example, they are all
\emph{integral} polynomials.  Moreover, they are \emph{alternating}.
One sees this latter property in \cite{mulase.waldron} for the
polynomials over $\HH$, which are obtained from those over $\RR$ by a
changing the parameter $n$ and including an overall sign for
contributions from surfaces of odd Euler characteristic, a phenomenon
Mulase--Waldron explain via a duality between the Gaussian orthogonal
and Gaussian symplectic matrix ensembles.

It would be very interesting to provide an algebraic model with $n$ as
a parameter that actually produces these polynomials.  Proving that
they are integral and alternating would also be interesting, as well
as providing a direct combinatorial interpretation of their
coefficients.

\begin{table}[htb]
\begin{center}
\begin{tabular}{|c||l|}
\hline
$k$& $C_{\OO} (n,k)$\\
\hline
$2$&$4 n^2 - 3 n$\\
$4$& $32 n^3 - 60 n^2 + 31 n $\\
$6$& $299 n^4 - 930 n^3 + 1081 n^2 - 435 n$\\
$8$& $5992 n^5 - 26577 n^4 + 50942 n^3 - 46875 n^2
+ 16728 n $\\
\hline
\end{tabular}
\medskip
\caption{Octonionic trace polynomials\label{tab:oct}}
\end{center}
\end{table}

\end{remark}

\section{The perturbation series}\label{s:aa:perturb}

\subsection{}

Let $t, g_{3}, g_{4},\dotsc$ be indeterminates.  We regard the $g_{k}$
as deformation parameters, and package them together into a vector $\g
= (g_{3},g_{4},\dotsc)$.  Let $\m = (m_{3},m_{4},\dotsc )\in
\prod_{k\geq 3}\ZZ_{\geq 0}$ be a vector of multiplicities; we assume
$m_{k} = 0$ for all sufficiently large $k$.  Then we write $\g^{\m}$
for the monomial $\prod g_{k}^{m_{k}}$. Let $N (\m) = \sum k m_{k}$.

We consider the 
pertubation series
\[
F (\g, t, X)  = \exp \bigl(\sum_{k \geq 3}(g_{k}\Tr X^{k})t^{k}\bigr)
\]
and the expectation 
\begin{equation}\label{eq:perturb2}
\ipA{F (\g,t,X)} \in \QQ [g_{3},g_{4},\dotsc ]\fps{t}.
\end{equation}
Our goal is to compute the coefficient of $t^{N}$ in
\eqref{eq:perturb2} in terms of gluings of polygons as in Theorem
\ref{th:tracexk}.  Clearly this coefficient is a homogeneous polynomial of degree
$N$ in the monomials $\g^{\m}$, where $N = N (\m)$, so it suffices to
compute the coefficient of $\g^{\m}t^N$.  To state the answer, we need
to extend our previous notation.

\subsection{}

Let $\Pi_{\m}$ be the disjoint union of polygons $ \coprod
\Pi_{k}^{*}$, where we take $m_{k}$ copies of $\Pi_{k}^{*}$.  We write
$\Pi \in \Pi_{\m}$ to mean that $\Pi$ is a connected component of
$\Pi_{\m}$.  Let $\sE$ be the set of all edges of $\Pi_{\m}$, and for
each $\Pi \in \Pi_{\m}$ we denote its set of edges by $\sE (\Pi)$.

Let $\pi$ be an oriented gluing of the edges of $\Pi_{\m}$, where as
before we allow both twisted and untwisted identifications. Note that
$\pi$ will in general glue together edges in different connected
components.  We say that a map $f\colon \sE \rightarrow U$ is
compatible with $\pi$ if it satisfies the extensions of our previous
conditions:
\begin{enumerate}
\item For each connected component $\Pi\in \Pi_{\m}$, the product 
$
\prod_{E\in \sE (\Pi)} f (E)
$
must be real-valued.
\item If $E,E'$ are identified by $\pi$, then $f (E) = f (E')$.
\end{enumerate}
We define the sign $\alpha (\pi ,f)$ exactly as before, and put 
\[
\Omega (\pi ,f) = \alpha (\pi ,f) \prod_{\Pi \in \Pi_{\m} } [f (\sE
(\Pi))]_{\fn},
\]
where we write $[f (\sE
(\Pi))]_{\fn}$ to mean $[f (E_{1})\cdots f (E_{l})]_{\fn}$, where
$E_{1},\dotsc ,E_{l}$ are the  edges  $\Pi$, again arranged clockwise
starting from the distinguished vertex. 

Finally, we define the group $\Aut \pi$ of automorphisms of $\pi$ to
be the group induced from permuting the connected components.  Note
that cyclic rotation of the connected components is not allowed, since
such symmetries do not preserve the distinguished vertex.  We can now
state our theorem:

\begin{thm}\label{th:perturb}
The coefficient of $\g^{\m}$ in the coefficient of $t^{N (\m)}$ is 
\[
2^{-N (\m )/2} \sum_{\pi} \sum_{f} \frac{\Omega (\pi ,f)}{|\Aut \pi |} 3^{N (\pi)},
\]
where the first sum is taken over all oriented gluings  of the edges of
$\Pi_\m $, and the second sum is taken over all functions $f\colon \sE
\rightarrow U$ compatible with $\pi $. 
\end{thm} 

\begin{proof}
The proof is a simple application of the exponential formula for
generating functions together with Theorem \ref{th:tracexk}.  The only
subtlety is the point that one considers all functions $f$ satisfying
the condition that $\prod f (E)$ be real-valued on each connected
component of $\Pi_{\m}$.  But this follows from the formula for the
trace given in Proposition \ref{prop:trpol}.
\end{proof}

\begin{ex}\label{ex:3.3}
We give an example to show how to apply Theorem \ref{th:perturb} and
compute the coefficient of $g_{3}^{2}$.  Thus $\m = (2,0,\dotsc)$ and
$\Pi_{\m}$ is two triangles, and the only automorphism is interchanging
the two components.  Thus the coefficient is $\ipA{(\Tr
X^{3})^{2}}/2$.  We will show how to compute the expectation
$\ipA{(\Tr X^{3})^{2}}$.

There are $5\cdot 3 = 15$ different gluings of the edges of
$\Pi_{\m}$.  Of these there are three essentially different types
(Figure~\ref{fig:twotris}); the first occurs $9$ times, and the other
two $3$ times each.  For each pairing of the edges, we must choose
whether we glue with a twist or not.  We can denote the twists by
vectors in $(\ZZ /2\ZZ )^{3}$.  Thus $010$ means $a$ and $c$ are glued
without a twist, whereas $b$ is glued with a twist.

Consider the gluing of type (I).  There are four possibilities
for the map $f$: (i) $f$ is identically $1$, (ii) $f$ takes $b$ to $1$
and $a, c$ to imaginary units, (iii) $f$ takes $a, b$ to $1$ and $c$
to any imaginary unit, and (iv) $f$ takes $b, c$ to $1$ and $a$ to any
imaginary unit.  The contributions of each of these, along with the
quantity $N (\pi)$, is summarized in Table \ref{tab:I}.  One sees that
a gluing of type (I) contributes $32\cdot 3^{3} - 32\cdot 3^{2} +
8\cdot 3$.

Next consider type (II); the contributions are summarized in Table
\ref{tab:II}.  This time there are three different possibilites: (i)
$f$ is identically $1$, (ii) $f$ takes one of $a,b,c$ to $1$ and the
other two to imaginary units, and (iii) $f$ takes $a, b, c$ to three
different imaginary units. Note that real-valuedness forces that in
(iii), the images of $f$ generate a subalgebra isomorphic to $\HH$.
This means that all products for this type are associative, so there
is no need to average over bracketings.  Thus the $7$ in column (iii)
corresponds to the $7$ lines in the Fano plane, and the $6$
corresponds to the ways to order the three imaginary units in a given
line.  Note also that type (iii) contributions do not arise in Theorem
\ref{th:tracexk}, since there we do not have two different sets of
edges to label with matching units.  The result is that a gluing of
type (II) contributes $16\cdot 3^{3}-24\cdot 3^{2} + 16\cdot 3$.

Finally consider type (III).  We have the same three possibilities for
$f$ as in type (II).  The contributions are essentially the same as in
Table \ref{tab:II}, except that each twisting datum should be replaced
by its complement.  We give the result in Table \ref{tab:III}.

To get the  final result we sum these contributions and divide by
$2^{N (\m)/2} =8$.  We obtain $9 (32\cdot 3^{3} - 32\cdot 3^{2} +
8\cdot 3)/8 + (3+3) (16\cdot 3^{3}-24\cdot 3^{2} + 16\cdot 3)/8 =
2709$.  One can verify that this agrees with a direct evaluation of
the integral defining $\ipA{(\Tr X^{3})^{2}}$.

\begin{figure}[htb]
\psfrag{a}{$a$}
\psfrag{b}{$b$}
\psfrag{c}{$c$}
\psfrag{i}{(I)}
\psfrag{ii}{(II)}
\psfrag{iii}{(III)}
\begin{center}
\includegraphics[scale=0.5]{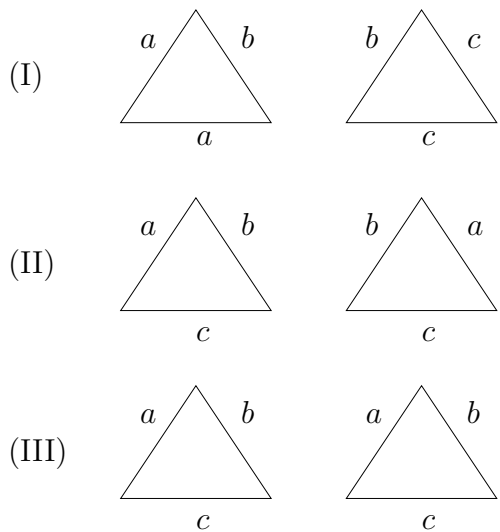}
\end{center}
\caption{Gluings of two triangles\label{fig:twotris}}
\end{figure}
\end{ex}

% gluings of type I
\begin{table}[htb]
\begin{center}
\begin{tabular}{|c|c||c|c|c|c|}
\hline
$N (\pi)$&twisting&(i) all 1&(ii) one 1&(iii) two 1s on left&(iv) two 1s on right\\
\hline\hline
$3^{3}$&$000$&	1& $7^{2}$& $7$ &$7$\\
\hline
$3^{2}$&$100$&	1& $-7^{2}$& $7$ &$-7$\\
$3^{3}$&$010$&	1& $7^{2}$& $7$ &$7$\\
$3^{2}$&$001$&	1& $-7^{2}$& $-7$ &$7$\\
\hline
$3^{2}$&$110$&	1& $-7^{2}$& $7$ &$-7$\\
$3^{1}$&$101$&	1& $7^{2}$& $-7$ &$-7$\\
$3^{2}$&$011$&	1& $-7^{2}$& $-7$ &$7$\\
\hline
$3^{1}$&$111$&	1& $7^{2}$& $-7$ &$-7$\\
\hline
\end{tabular}
\end{center}
\caption{Contributions of gluings of type (I)\label{tab:I}}
\end{table}

% gluings of type II
\begin{table}[htb]
\begin{center}
\begin{tabular}{|c|c||c|c|c|}
\hline
$N (\pi)$&twisting&(i) all 1&(ii) one 1&(iii) 3 different units\\
\hline\hline
$3^{3}$&$000$&	1& $3\cdot 7$& $6\cdot 7$\\
\hline
$3^{2}$&$100$&	1& $-7$& $-6\cdot 7$\\
$3^{2}$&$010$&	1& $-7$& $-6\cdot 7$\\
$3^{2}$&$001$&	1& $-7$& $-6\cdot 7$\\
\hline
$3^{1}$&$110$&	1& $-7$& $6\cdot 7$\\
$3^{1}$&$101$&	1& $-7$& $6\cdot 7$\\
$3^{1}$&$011$&	1& $-7$& $6\cdot 7$\\
\hline
$3^{1}$&$111$&	1& $3\cdot $& $-6\cdot 7$\\
\hline
\end{tabular}
\end{center}
\caption{Contributions of gluings of type (II)\label{tab:II}}
\end{table}

% gluings of type III
\begin{table}[htb]
\begin{center}
\begin{tabular}{|c|c||c|c|c|}
\hline
$N (\pi)$&twisting&(i) all 1&(ii) one 1&(iii) 3 different units\\
\hline\hline
$3^{1}$&$000$&	1& $3\cdot 7$& $-6\cdot 7$\\
\hline
$3^{1}$&$100$&	1& $-7$& $6\cdot 7$\\
$3^{1}$&$010$&	1& $-7$& $6\cdot 7$\\
$3^{1}$&$001$&	1& $-7$& $6\cdot 7$\\
\hline
$3^{2}$&$110$&	1& $-7$& $-6\cdot 7$\\
$3^{2}$&$101$&	1& $-7$& $-6\cdot 7$\\
$3^{2}$&$011$&	1& $-7$& $-6\cdot 7$\\
\hline
$3^{3}$&$111$&	1& $3\cdot $& $6\cdot 7$\\
\hline
\end{tabular}
\end{center}
\caption{Contributions of gluings of type (III)\label{tab:III}}
\end{table}

\begin{remark}
As in Remark \ref{rem:previous}, one can consider replacing $3^{N
(\pi)}$ by $n^{N (\pi)}$ and use our gluing calculus to produce
polynomials in $n$ for the mixed moments.  For example for $\ipA{(\Tr
X^{3})^{2}}$, one finds the result is $192n^{3}-324n^{2}+147n$.
Again, these polynomials appear to be integral, if one ignores the
denominators coming from the exponential series.  They also appear to
have alternating coefficients, just as in Table \ref{tab:oct}.  We
have no explanation for this fact.
\end{remark}

\part{The spin factor.}\label{part:sf}

\section{Background}\label{s:sf:background}

\label{ss:sf:notation}
\subsection{} We begin by recalling the definition of the \emph{spin
factor} $\LL = \LL_{1,n}$.  As an $\RR$-vector space $\LL$ is $\RR
\times \RR^{n}$.  Write elements $\x \in \LL$ as pairs $(x_{0},x)$,
where $x_{0}\in \RR$.  Then the Jordan product is defined by $\x
\jordan \y = (x_{0}y_{0} + x\cdot y, x_{0}y + y_{0}x)$, where $\cdot$
denotes the usual Euclidean dot product on $\RR^{n}$.  The trace map
$\Tr \colon \LL \rightarrow \RR$ is defined by $\x \mapsto x_{0}$.  We
define powers $\x ^{k}$ as in the case of $\AA$, through
\eqref{eq:powers}.

The algebra $\LL$ is called a spin factor because of its connection
with Clifford algebras \cite[\S 1.9]{mccrimmon}.  Let $W$ be a real
vector space of dimension $n$ with orthonormal basis $v_{i}$,
$i=1,\dotsc ,n$.  The Clifford algebra $C (W)$ is the unital
associative algebra generated by $W$ modulo the relations
$v_{i}^{2}=1$, and $v_{i}v_{j}=-v_{j}v_{i}$ for all
$i\not =j$.  The algebra $C (W)$ has dimension $2^{n}$, with an
additive basis given by $1$ and all expressions of the form
\[
v_{i_{1}}\cdots v_{i_{k}}, \quad  \text{where $1 \leq i_{1} < \dotsb < i_{k}
\leq n$ and $k=1,\dotsc ,n$}.
\]
The $n+1$-dimensional subspace spanned by $1$ and the $v_{i}$ does not
form an associative subalgebra, but it does inherit the structure of a
Jordan algebra as in \S \ref{ss:assoctojord}.  It is easy to check
that this algebra is exactly $\LL_{1,n}$.  Since $C
(W)$ can be realized as a subalgebra of the $2^{n}\times 2^{n}$
symmetric matrices over $\RR$, this means that $\LL$ can be viewed as
a sub-Jordan algebra of $V_{\RR}$.  For more details we refer to
\cite{mccrimmon}.

\section{Computation of the basic trace integral}\label{s:sf:traceint}

\subsection{}\label{ss:powers} The goal of this section is to compute
combinatorially the expectations $\ipL{\Tr \x ^{k}}$ of the trace
monomials.  As we shall see, when $k$ is fixed and $n$ is taken as a
parameter, we obtain a polynomial $C_{\LL} (n,k)$ in $n$, just as in the classical
case of Hermitian matrices.  We shall also see that the combinatorial
model is one dimensional, as in the case of Feynman diagrams.  In
fact, we give two closely related models.  The first allows one to
quickly evaluate $C_{\LL} (n,k)$, whereas the second makes it easy to
incorporate automorphisms in the model.

A first step is to compute the trace polynomials explicitly.  We use
the notation of \S \ref{ss:sf:notation}, along with the convention
that for the vector part $x$ of $\x = (x_{0},x)$, the symbol $x^{k}$
denotes
\begin{itemize}
\item the scalar $(x \cdot x)^{k/2}$ if $k$ is even, and 
\item the vector $(x \cdot x)^{(k-1)/2} x$ if $k$ is odd.
\end{itemize}
Then we have the following result; we omit the easy proof by induction.

\begin{prop}\label{prop:tracepolys}
We have $\x^{k} = (z_{0}, z)$, where
\begin{equation}\label{eq:sf:tracepoly}
z_{0} = \sum_{\substack{i+j=k\\
\text{$j$ even}}} \binom{k}{i}x_0^{i} x^{j}, \quad z = \sum_{\substack{i+j=k\\
\text{$j$ odd}}} \binom{k}{i}x_0^{i} x^{j}.
\end{equation}
\end{prop}

\subsection{}\label{ss:bbs}
Using Proposition \ref{prop:tracepolys}, it is easy to directly
compute the expectation $\ipL{\Tr \x^{k}}$.  Our goal is to give a
combinatorial description that makes the role of the parameter $n$
more apparent.  

Let $S$ be a finite set of points labelled by $\{1,\dotsc
,k \}$, where we assume $k$ is even.  Let $S = S_{0} \cup S_{b}$ be a
partition of $S$ into two subsets, each of even order.  We take the
points in $S_{b}$, order them by their labels, and join consecutive
ones by an edge.  The result is a collection of points and edges that
we call a \emph{barbell structure} on $S$, with the set of edges being
called the \emph{barbells}.  

Let $\beta$ be a barbell structure on $S$.  We define a \emph{pairing} $\pi (\beta)$
of $\beta$ to be a pairing of the elements of $S_{0}$ and the elements
of $S_{b}$.  In other words, we can freely join any elements of the
two parts together in pairs, but we cannot join an element in one part
to an element in the other.  Each pairing produces a union of
edges (in $S_{0}$) and circles (in $S_{b}$).  Let $N (\pi)$ be
the number of connected components of $\pi$ in $S_{b}$.  Finally
define a \emph{coloring} of $\pi $ to be an assignment of $\{1,\dotsc
,n \}$ to each connected component in $S_{b}$.

\begin{thm}\label{th:sf:expectation}
Let $\LL$ be the spin factor $\LL_{1,n}$.  For $k$ odd we have
$\ipL{\Tr \x^{k}} = 0$.  For $k$ even we have
\begin{equation}\label{eq:sf:expectation}
\ipL{\Tr \x^{k}} = \sum_{\beta} \sum_{\pi} n^{N (\pi)},
\end{equation}
where $\beta$ ranges over all barbell structures on $\{1,\dotsc ,k \}$
and $\pi$ ranges over all pairings of $\beta$.
\end{thm}

\begin{proof}
Proposition \ref{prop:tracepolys} shows that $\Tr \x^{k}$ is an odd
function of $x_{0}$ if $k$ is odd, so certainly $\ipL{\Tr \x^{k}} = 0$
in that case.  So suppose $k$ is even and write $k=i+j$ with $i,j$
even.  There are clearly $\binom{k}{i}$ barbell structures on
$\{1,\dotsc ,k \}$ with $|S_{0}| = i$, and any such one will
contribute a factor of $w (k)$ coming from the pairings in $S_{0}$.
Thus the result will follow if we can show
\[
\ipL{x^{j}} = \Bigl\langle\bigl(\sum_{p=1}^{n}
x_{p}^{2}\bigr)^{j/2}\Bigr\rangle_{\LL } = \sum_{\pi'} n^{N (\pi')},
\]
where $\pi'$ ranges over the colored pairings of a \emph{fixed}
collection of $j/2$ barbells.

This can be seen as follows.  Use $\pi '$ to pair the endpoints of
each barbell.  Consider labeling the endpoints of each barbell with a
variable $x_{p}$, $p=1,\dotsc ,n$ such that the same variable appears
at either end, and such that the variable appearing along each
connected component is constant.  There are clearly $n^{N (\pi ')}$
such assignments of variables.  Each one corresponds to a monomial
produced by multiplying out $(\sum_{p=1}^{n} x_{p}^{2})^{j/2}$, where
the variable $x_{p}$ from the $l$th factor is placed on the $l$th
barbell.  This completes the proof.
\end{proof}

\begin{ex}\label{ex:sf:sixth}
In Figure \ref{fig:tracesix} we give the full computation of $\ipL{\Tr
\x^{6}}$.  There are barbell structures with $|S_b| = 0,2,4,6$.
\end{ex}

\begin{figure}[htb]
\psfrag{contribution}{Contribution}
\psfrag{s0}{$S_{0}$}
\psfrag{sb}{$S_{b}$}
\psfrag{1}{(i)}
\psfrag{2}{(ii)}
\psfrag{3}{(iii)}
\psfrag{e}{$\emptyset $}
\psfrag{w6}{$w (6)= 15$}
\psfrag{w4}{$w (4)\binom{6}{4}n = 45n$}
\psfrag{w22}{(ii) $w (2)\binom{6}{2} 2n = 30n$}
\psfrag{w21}{(i) $w (2)\binom{6}{2} n^{2} = 15n^{2}$}
\psfrag{w02}{(ii) $2\binom{3}{1}  n^{2} = 6n^{2}$}
\psfrag{w01}{(i) $n^{3}$}
\psfrag{w03}{(iii) $8n$}
\begin{center}
\includegraphics[scale=0.30]{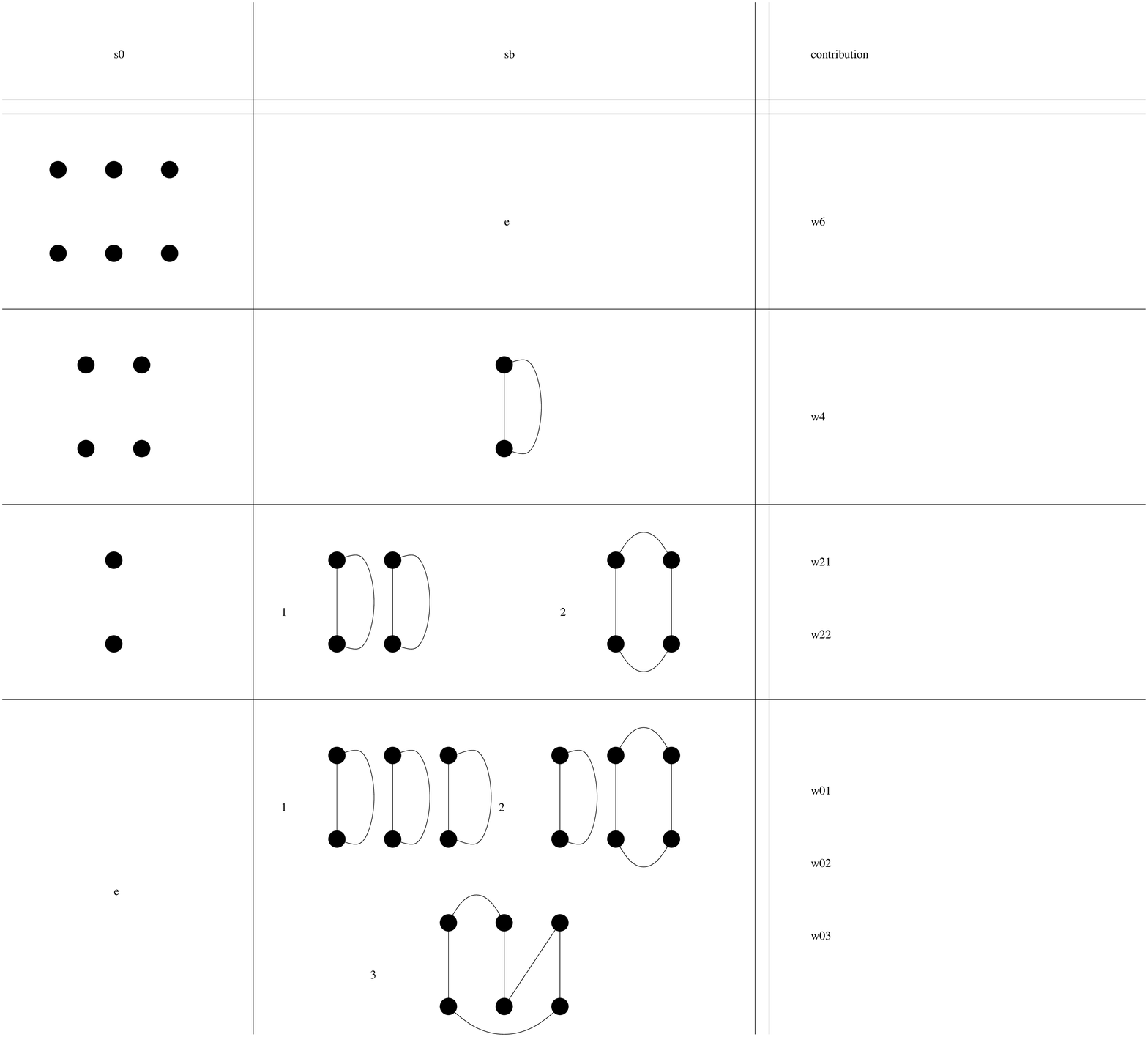}
\end{center}
\caption{Computing $\ipL{\Tr \x^{6}}.$ The result is $n^3 + 21n^2 + 83n + 15$.\label{fig:tracesix}}
\end{figure}

\subsection{}
For $k$ even let $C_{\LL} (n,k)$ be the polynomial \eqref{eq:sf:expectation}.
The polynomials $C_{\LL} (n,k)$ can be seen in Table
\ref{tab:polynomials}.  It is clear that they are monic, and the
constant terms are the Wick numbers.  The coefficients of
the codegree one terms
\[
1,8,21,40,65,96,\dotsc 
\]
are the \emph{octagonal numbers} \cite[\texttt{A000567}]{oeis}; these
are the analogue of the triangular numbers, in which one arranges dots
in an octagon (Figure \ref{fig:octagonal}).  This can be seen as
follows.  Write $S = S_{0} \cup S_{b}$ and suppose $|S|=k$ with $k=2m$
even.  There are two ways a paired barbell structure can contribute to
this degree.  Either $S_{0} = \emptyset$ or $|S_0| = 2$.  In the
former case we have $m$ barbells, and we must choose two of them to
pair into a connected component (the other $m-1$ must be paired to
themselves).  There are two pairings giving one connected component,
so there are $\binom{m}{2}\cdot 2$ paired barbell structures of this
type.  In the latter case we have to pick which two points will go to
$S_{0}$, so there are $\binom{2m}{2}$ paired barbell structures of
this type.  Hence altogether this coefficient is
\begin{equation}\label{eq:codegree1}
\binom{m}{2}\cdot 2 + \binom{2m}{2}.
\end{equation}
Now in Figure \ref{fig:octagonal} we can shave off two triangles to
yield a hexagon as in Figure \ref{fig:shave}.  Since the hexagonal
number is well known to be $\binom{2m}{2}$, this gives
\eqref{eq:codegree1}.  Apart from these sequences of coefficients, not
much else seems to be known about these polynomials.
% The coefficients of the codegree two terms 
% \[
% 3,83,422,1310,3145,6433,11788,\dotsc 
% \]
% satisfy the polynomial $3+121/6n+34n^2+64/3n^3+9/2n^4$ (from OEIS
% superseeker, using finite differences)

\begin{figure}[htb]
\centering
\subfigure[The first three octagonal numbers\label{fig:octagonal}]{\includegraphics[scale=0.25]{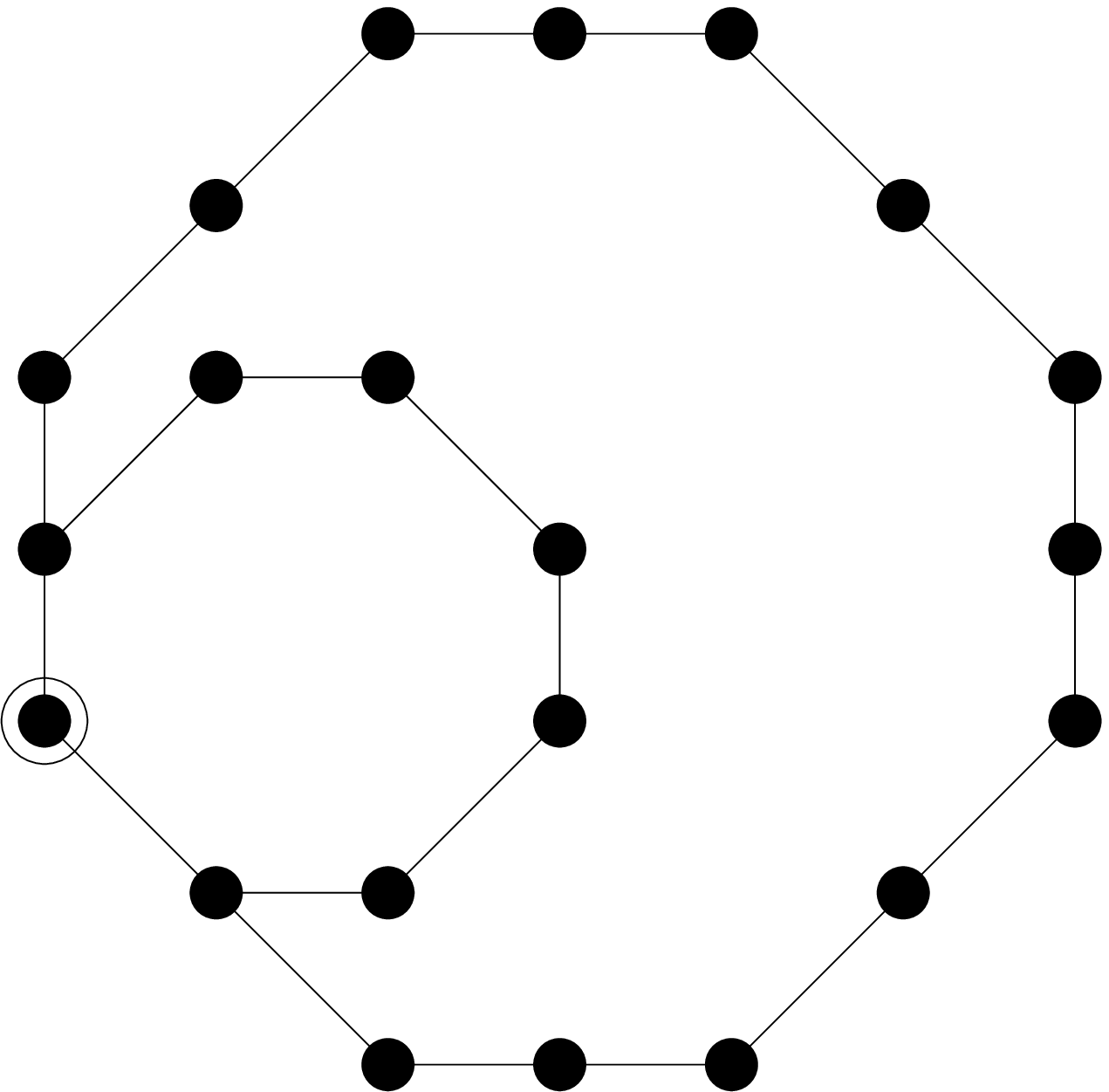}}
\quad\quad\quad\quad\quad\quad
\subfigure[Shaving the octagon\label{fig:shave}]{\includegraphics[scale=0.25]{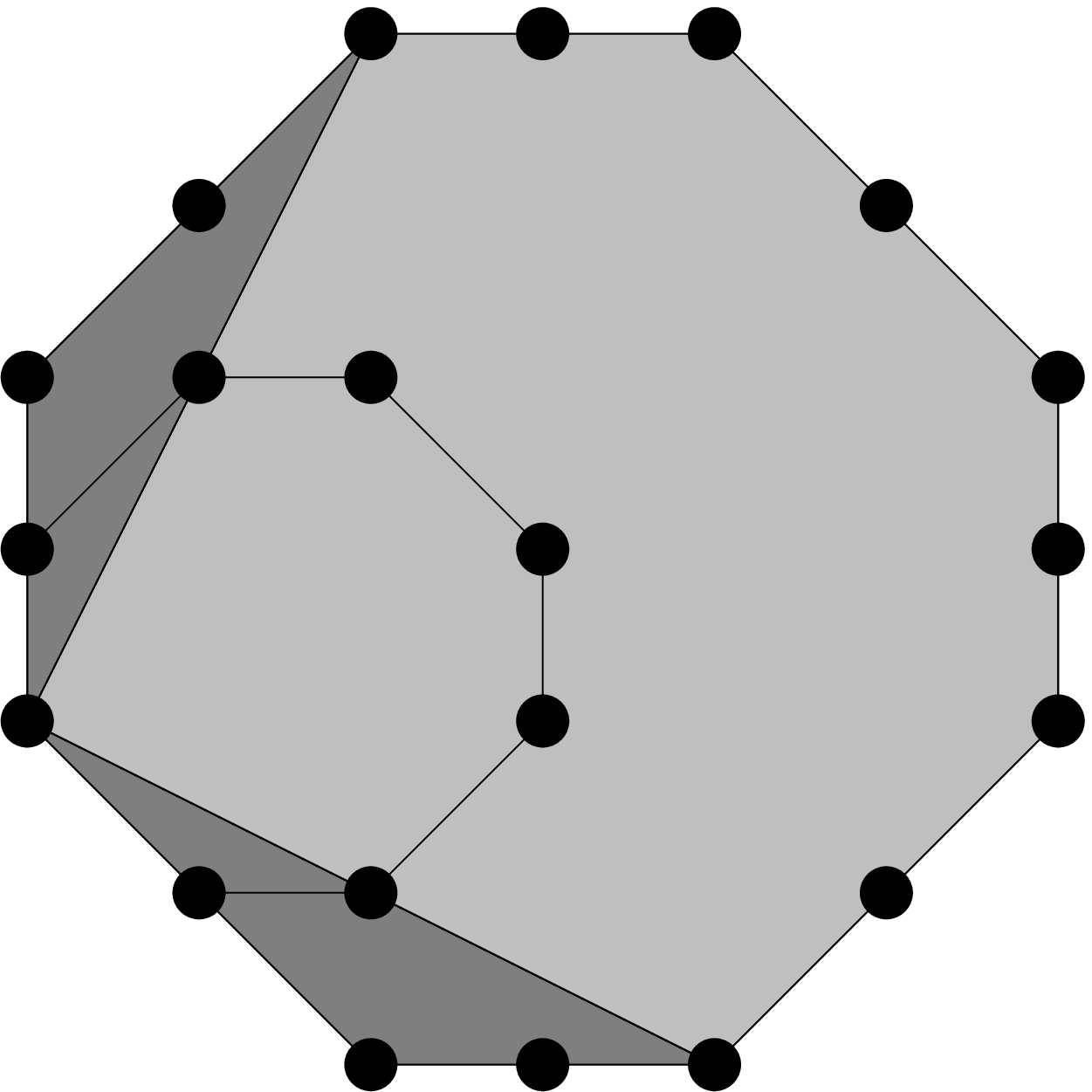}}
\caption{\label{fig:octashave}}
\end{figure}

\subsection{}\label{ss:gfs} 
We conclude this section by showing how the polynomials $C_{\LL}
(n,k)$ can be easily computed using standard techniques of generating
functions.  First, we define the \emph{even Wick numbers} $w_{e} (k)$
by $w_{e} (k) = 0$ if $k$ is odd, and $k!!  := k(k-2) (k-4)\dotsb 2$
if $k$ is even.  Given $k$ barbells, there are $w_{e} (k)$
pairings of the ends that yield one connected component.  Then the
exponential power series
\begin{align*}
A &= \exp (\sum_{k\geq 1} w_{e} (2k) n \frac{x^{k}}{k!}) \\
&= 1
 + n x
 + \bigl(\frac{n^{2}}{2}
 + n\bigr)  x^2
 + \bigl(\frac{n^{3}}{6}
 + n^2
 + \frac{4n}{3} \bigr)  x^3
 + \bigl(\frac{n^{4}}{24}
 + \frac{n^{3}}{2}
 + \frac{11n^{2}}{6}
 + 2 n\bigr)  x^4
 + \dotsb 
\end{align*}
gives the generating function of all possible barbell pairings.  

To get the polynomials in Table \ref{tab:polynomials}, we need to
complete the paired barbells into paired barbell structures.  As a
first step we need to tweak $A$.  Let $A_{l}$ be the result of
applying the series Laplace transform to $A$,\footnote{This transform
takes $\sum a_{k}x^{k}/k!$ to $\sum a_{k}x^{k}$.} and let $A'$ be the
result of replacing $x$ with $x^{2}$ in $A_{l}$ and then convolving
with the exponential series.  The result is
\[
A' = 1
 + \frac{n}{2}x^2
 + \bigl(\frac{n^{2}}{24}
 + \frac{n}{12} \bigr)  x^4
 + \bigl(\frac{n^{3}}{720}
 + \frac{n^{2}}{120}
 + \frac{n}{90} \bigr)  x^6+ \dotsb.
\]
This gives the same data as the exponential series $A$ for the barbell
pairings, but now the polynomials are placed in the correct (even)
degrees.

Finally we can incorporate the pairings in $S_{0}$.  Let $B$ be the
exponential generating function of the Wick numbers $w (k)$:
\begin{align*}
B &= \sum_{k\geq 0} w (k)\frac{x^{k}}{k!}\\
&=1
 + \frac{1}{2} x^2
 + \frac{1}{8} x^4
 + \frac{1}{48} x^6
 + \frac{1}{384} x^8
 + \frac{1}{3840} x^{10}+ \dotsb .
\end{align*}
The product $A'\cdot B$ then gives all ways to break up $S$ into
$S_{0} \cup S_{b}$ and to pair, along with the data of the number of
connected components obtained in $S_{b}$.  If we take the Laplace
transform of $A'\cdot B$, we obtain the ordinary generating function
of the polynomials $C_{\LL} (n,k)$:
\begin{equation}\label{eq:sf:Cgenfn}
(A'\cdot B)_{l} = 1
 + \left(n
 + 1\right)  x^2
 + \left(n^2
 + 8 n
 + 3\right)  x^4
 + \left(n^3
 + 21 n^2
 + 83 n
 + 15\right)  x^6
 + \dotsb 
\end{equation}
\begin{table}[htb]
\begin{center}
\begin{tabular}{|c||p{400pt}|}
\hline
$k$&$C_{\LL} (n,k)$\\
\hline\hline
0&$1$\\
2&$n + 1$\\
4&$n^2 + 8n + 3$\\
6&$n^3 + 21n^2 + 83n + 15$\\
8&$n^4 + 40n^3 + 422n^2 + 1112n + 105$\\
10&$n^5 + 65n^4 + 1310n^3 + 9310n^2 + 18609n + 945$\\
12&$n^6 + 96 n^5 + 3145 n^4 + 42720 n^3 + 231259 n^2 + 377664 n +
10395$
\\
14&$n^7 + 133 n^6 + 6433 n^5 + 141925 n^4 + 1466059 n^3 + 6476407 n^2
+ 9071187 n + 135135$
\\
16&$n^8 + 176 n^7 + 11788 n^6 + 383600 n^5 + 6424054 n^4 + 53966864
n^3 + 203378412 n^2 + 252726480 n + 2027025$\\
18&$n^9 + 225 n^8 + 19932 n^7 + 897372 n^6 + 22132614 n^5 + 300621510
n^4 + 2144046428 n^3 + 7109593308 n^2 + 8031454785 n + 34459425 $\\
20&$n^{10} + 280 n^9 + 31695 n^8 + 1885920 n^7 + 64273818 n^6 +
1283152080 n^5 + 14746708430 n^4 + 92004426080 n^3 + 274591498581 n^2
+ 287095866840 n + 654729075 $\\
 \hline
\end{tabular}
\end{center}
\medskip
\caption{The expectations $\ipL{\Tr \x^{k}}$, as a function of $n$.\label{tab:polynomials}}
\end{table}

\section{Automorphisms and connected structures}\label{s:sf:auts}

\subsection{}\label{ss:barbellgroup}
Consider the generating function \eqref{eq:sf:Cgenfn} of the
polynomials $C_{\LL} (n,k)$:
\begin{equation}\label{eq:sf:Cgenfn2}
 1 + \left(n + 1\right) x^2 + \left(n^2 + 8 n + 3\right) x^4 +
\left(n^3 + 21 n^2 + 83 n + 15\right) x^6 + \dotsb
\end{equation}
In this section we modify \eqref{eq:sf:Cgenfn2} by taking into account
symmetries of barbell diagrams.  Let $k = 2m$ be even.  Let $\Xi_{k}$
be the set $\{1,\dotsc ,k \}$ with barbells drawn between the pairs
$\{1,2 \}$, $\{3,4 \}$, \dots, $\{k-1,k \}$.  Let $B_{m}$ be the
wreath product $\ZZ /2\ZZ \wr S_{m}$ of order $2^{m}m!$ We call
$B_{m}$ the \emph{barbell group}, and think of it as acting on
$\Xi_{m}$ by permuting the barbells and flipping them independently.
Thus if we label the points in $\Xi_{k}$ as above, then we can
identify $B_{m}$ with the subgroup of $S_{k}$ generated by the
transpositions $(1,2)$, $(3,4)$, \dots , $(k-1,k)$ and the products
$(1,3) (2,4)$, $(3,5) (4,6)$, \dots , $(k-2,k) (k-3,k-1)$ (see Figure
\ref{fig:barbellgroup}).

\begin{figure}[htb]
\psfrag{1}{$1$}
\psfrag{2}{$2$}
\psfrag{3}{$3$}
\psfrag{4}{$4$}
\psfrag{5}{$5$}
\psfrag{6}{$6$}
\psfrag{km1}{$k-1$}
\psfrag{k}{$k$}
\psfrag{d}{$\dotsb $}
\begin{center}
\includegraphics[scale=0.25]{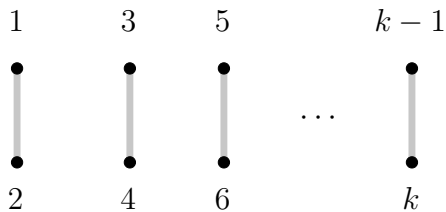}
\end{center}
\caption{The barbell group $B_{m}$, where $m=k/2$, permutes the barbells and flips
them independently.\label{fig:barbellgroup}}
\end{figure}

Our goal is to give a combinatorial meaning to the modified
generating function
\begin{equation}\label{eq:modified}
 B (x) = 1 + \frac{1}{2}\left(n + 1\right) x^2 + \frac{1}{8}\left(n^2
+ 8 n + 3\right) x^4 + \frac{1}{48}\left(n^3 + 21 n^2 + 83 n +
15\right) x^6 + \dotsb,
\end{equation}
in which each polynomial $C_{\LL} (n,k)$ is divided by the order of
$B_{k/2}$.  In particular, we want to express the coefficient of
$x^{k}$ as a sum over various paired configurations of $k$ barbells up
to isomorphism, where each configuration is weighted by the inverse of
the order of its automorphism group.  This is analogous to the usual
Feynman calculus, which expresses coefficients of certain power series
as sums over certain graphs weighted by the inverses of the orders of
their automorphism groups.  As a simple example, consider the power series
(cf.~\eqref{eq:realpairing})
\begin{equation}\label{eq:fseries}
\ip{\exp (tx^{4}/4!)}_{\RR} = 1 + c_{1}t + c_{2}t^{2} + \dotsb .
\end{equation}
Then we have
\[
c_{j} = \sum_{\substack{\Gamma \in G_{j}}} \frac{1}{|\Aut \Gamma |},
\]
where the sum is taken over all graphs with $j$ vertices of degree
$4$, and where the automorphisms are induced by permuting vertices and
edges (including flips of loops).  For instance, $c_{1} = 1/8$ and $c_{2} =
35/384$ (cf.~Figure \ref{fig:fourgraphs}).  For more details, we refer
to \cite[\S 3.2]{etingof} and \cite[Ch. 9]{mirror}.

\begin{figure}[htb]
\psfrag{eight}{$8$}
\psfrag{fortyeight}{$48$}
\psfrag{onetwentyeight}{$128$}
\psfrag{sixteen}{$16$}
\begin{center}
\includegraphics[scale=0.25]{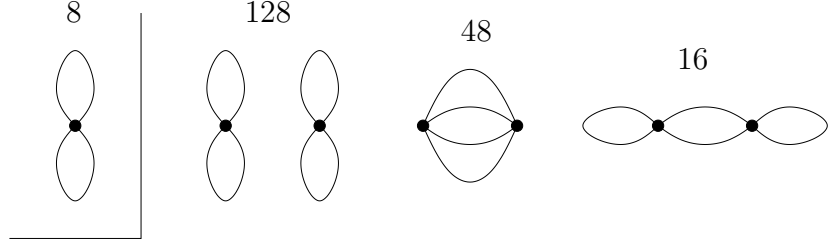}
\end{center}
\caption{Graphs and their numbers of automorphisms used to compute
$c_{1}=1/8$ and $c_{2} = 35/384$ in
\eqref{eq:fseries}.\label{fig:fourgraphs}}
\end{figure}

\subsection{} To carry this out, we give a slightly different model
for the terms contributing to the expectation.  We define a
\emph{barbell graph} to be a graph $\Gamma$ constructed as follows.
Begin with $\Xi_{k}$.  We partition the vertices into two sets
$S_{bl}$ and $S_{gr}$, where we color the vertices in $S_{bl}$
(respectively $S_{gr}$) black (respectively green).  We do this
arbitrarily; in particular the ends of a barbell need not be the same
color.  Then we make an arbitrary pairing of the ends of the barbells
that is compatible with the coloring.  This means we only pair black
to black and green to green; we never mix colors.  We say two barbell
graphs $ \Gamma, \Gamma ' $ are equivalent if we can carry $\Gamma$ to
$\Gamma '$ using the action of $B_{m}$.  Let $N (\Gamma)$ be the
number of connected components of $\Gamma$ that contain at least one
green vertex, and let $\Aut \Gamma \subset B_{m}$ be the subgroup of
automorphisms.

\begin{thm}\label{th:secondmodel}
Let $m=k/2$.  Then we have 
\begin{equation}\label{eq:orbits}
\frac{1}{|B_{m}|} C_{\LL} (n,k) = \sum_{\Gamma \in G (m)} \frac{n^{N (\Gamma)}}{|\Aut \Gamma |},
\end{equation}
where the $G (m)$ is the set of equivalence classes of barbell graphs
with $m$ barbells.
\end{thm}

\begin{proof}
Let $G^{*} (m)$ be the set of all barbell graphs with $m$ barbells,
without modding out by the action of $B_{m}$.  We will show
\begin{equation}\label{eq:noauts}
C_{\LL} (n,k) = \sum_{\Gamma \in G^{*} (m)} n^{N (\Gamma)},
\end{equation}
which implies \eqref{eq:orbits}.  Indeed, by definition the elements
of $G (m)$ are the orbits of $B_{m}$ in $G^{*} (m)$, and thus
\eqref{eq:orbits} follows from the orbit-stabilizer formula.  To prove
\eqref{eq:noauts}, we will show that the total contribution from the
barbell graphs $G^{*} (m)$ agrees with that from the paired barbell
structures in \S \ref{ss:bbs}.

Thus let $\beta = S_{0}\cup S_{b}$ be a barbell structure on
$\{1,\dotsc ,k \}$, and consider the fixed collection of barbells
$\Xi_{m}$ as in Figure~\ref{fig:barbellgroup}.  When we compute the
contribution of all pairings $\pi$ of $\beta$, the result has the form
$w P (n)$, where $w$ is the Wick number $w (|S_{0}|)$ and $P (n)$ is a
polynomial in $n$ of degree $|S_{b}|/2$.  We claim $\beta$ tells us
how to build a collection of barbell graphs giving the same total
contribution, and that by varying $\beta$ we obtain all graphs in
$G^{*} (m)$.

First, the partition $\beta$ tells us how to color the vertices in
$\Xi_{m}$: we color those with labels in $S_{0}$
(respectively, $S_{b}$) black (resp., green).  This determines the
sets $S_{bl}$ and $S_{gr}$.  Next, choose an arbitrary pairing of
$S_{bl}$.  We claim, once this pairing is fixed, that after adding
together the contributions coming from all pairings in $S_{gr}$ one
obtains $P (n)$.  This proves the result, since there are $w$ possible
choices of pairings in $S_{bl}$.

So let $\Gamma$ be the union of path graphs formed after pairing the
vertices in $S_{bl}$.  The connected components of $\Gamma$ either
have black or green endpoints.  The components with black
endpoints are irrelevant and can be ignored.  Those with green
endpoints either have no internal vertices or have all internal
vertices black.  Since there are $|S_{gr}|/2$ connected components
with green vertices, the contribution after all pairings of $S_{gr}$
are formed will be $P (n)$.  This shows that either model, barbell
structures or barbell graphs, produces the same expectation $\ipL{\Tr
\x^{k}}$, and completes the proof.
\end{proof}

\subsection{} The series $B (x)$ counts all the barbell pairings
divided by the orders of their automorphism group.  Just as in the
usual Feynman calculus, one can simplfy the computation by reducing to
the connected diagrams, since typically there are far fewer connected
than general diagrams.  In other words, one considers the generating
function
\[
B_{c} (x) := \log B (x).
\]
It turns out that this series has a particularly simple form:

\begin{thm}\label{th:sf:connected}
We have 
\begin{align*}
B_{c} (x)  &= \frac{1}{2} (n+1)x^{2} + \frac{1}{4} (3n+1)x^{4} + \frac{1}{6}
 (7n+1)x^{6} + \frac{1}{8} (15n+1)x^{8} + \dotsb \\
 &= \sum_{m\geq 1} \frac{1}{2m} ((2^{m}-1)n +1) x^{2m}.
\end{align*}
\end{thm}

\begin{proof}
First, using the barbell group we can carry any connected barbell
graph into a standard form: the pairings all connect the bottom of the
$i$th barbell to the top of the $(i+1)$st barbell for $1\leq i\leq m$.
We take these labels mod $m$, which means the bottom of the last
barbell is connected to the top of the first (cf.~Figure
\ref{fig:bb12}).  For the purposes of this proof, we will say that
such a connected barbell graph is \emph{standardly paired}.  If there
are no green vertices, then the automorphism group of this pairing has
order $2m$, and is the group $G\subset B_{m}$ generated by cyclic
permutation of the barbells and simultaneous flipping of all of them
about both axes.  This implies that the constant term of $x^{2m}$ is
$1/2m$.

Now we claim that the standardly paired connected barbell graphs with
$2m$ vertices are in bijection with subsets of $\{1,\dotsc, m \}$.
Indeed, let $I = \{i_{1},\dotsc ,i_{l} \}$ be a subset of $\{1,\dotsc
,m \}$.  Then we simply take the standardly paired barbell graph with
all vertices initially black, and color the bottom of barbell $i_{j}$
and the top of barbell $i_{j}+1$ green for $j=1,\dotsc ,l$.  This
graph clearly has $N (\Gamma) = 1$.  Since there are $2^{m}-1$ such
graphs with at least one green vertex, the result follows once we
divide out by the action of $G$.

\end{proof}

\begin{figure}[htb]
\begin{center}
\includegraphics[scale=0.25]{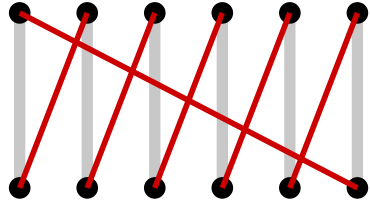}
\end{center}
\caption{A standardly paired connected barbell graph.\label{fig:bb12}}
\end{figure}

\begin{ex}\label{ex:trx6}
Figures \ref{fig:b60}--\ref{fig:b63} show the barbell graphs used in
the computation of the coefficient of $x^{6}$ in \eqref{eq:modified},
along with their contributions $n^{N (\Gamma)}/|\Aut \Gamma |$.  Thus
the graph $\Gamma$ in Figure \ref{fig:b60.1} has $|\Aut \Gamma | = 48$
and contributes to the constant term.  The connected graphs, which
contribute to the coefficient of $x^{6}$ in Theorem
\ref{th:sf:connected}, are indicated with a star $\star$.
\end{ex}

\begin{figure}[htb]
\centering
\subfigure[$1/48$\label{fig:b60.1}]{\includegraphics[scale=0.20]{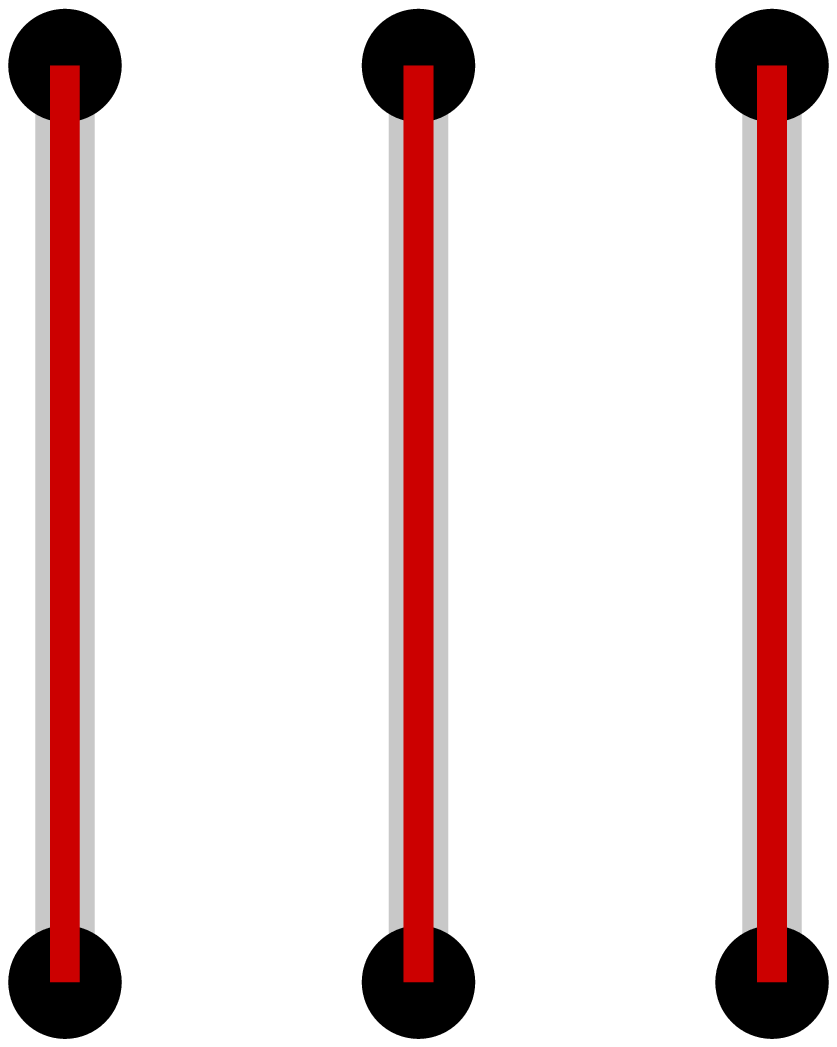}}
\quad\quad\quad 
\subfigure[$1/6$\,$\star$\label{fig:b60.2}]{\includegraphics[scale=0.20]{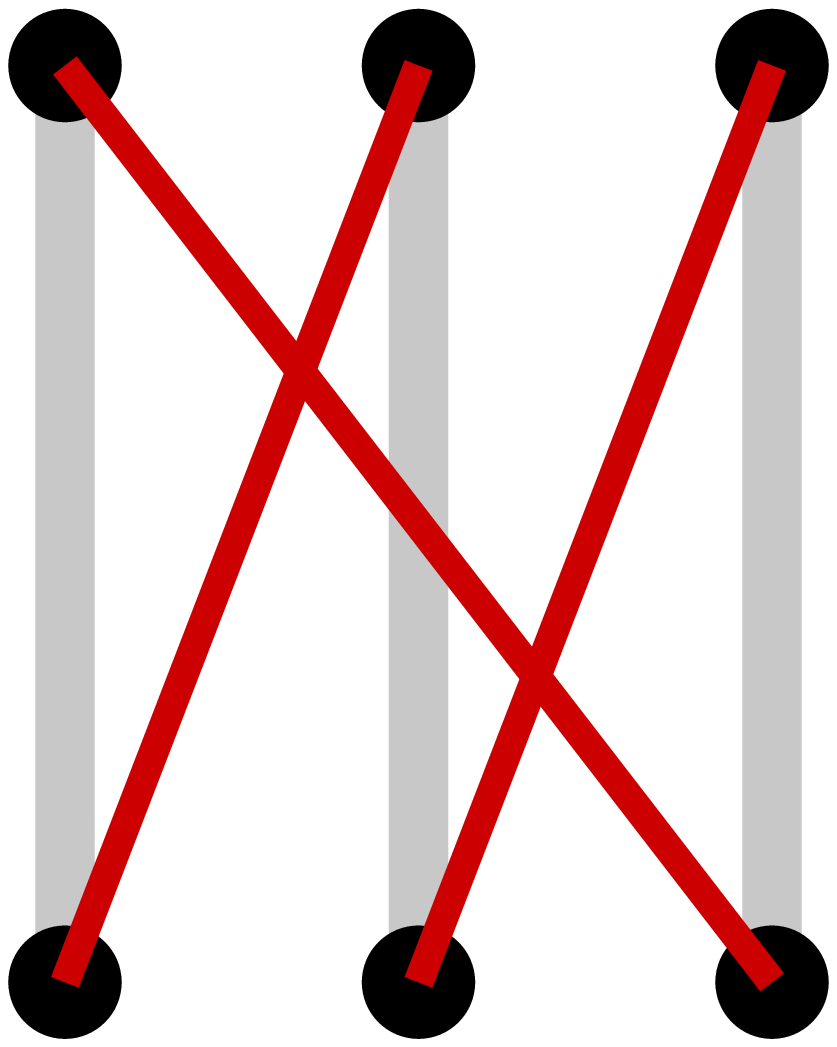}}
\quad\quad\quad 
\subfigure[$1/8$\label{fig:b60.3}]{\includegraphics[scale=0.20]{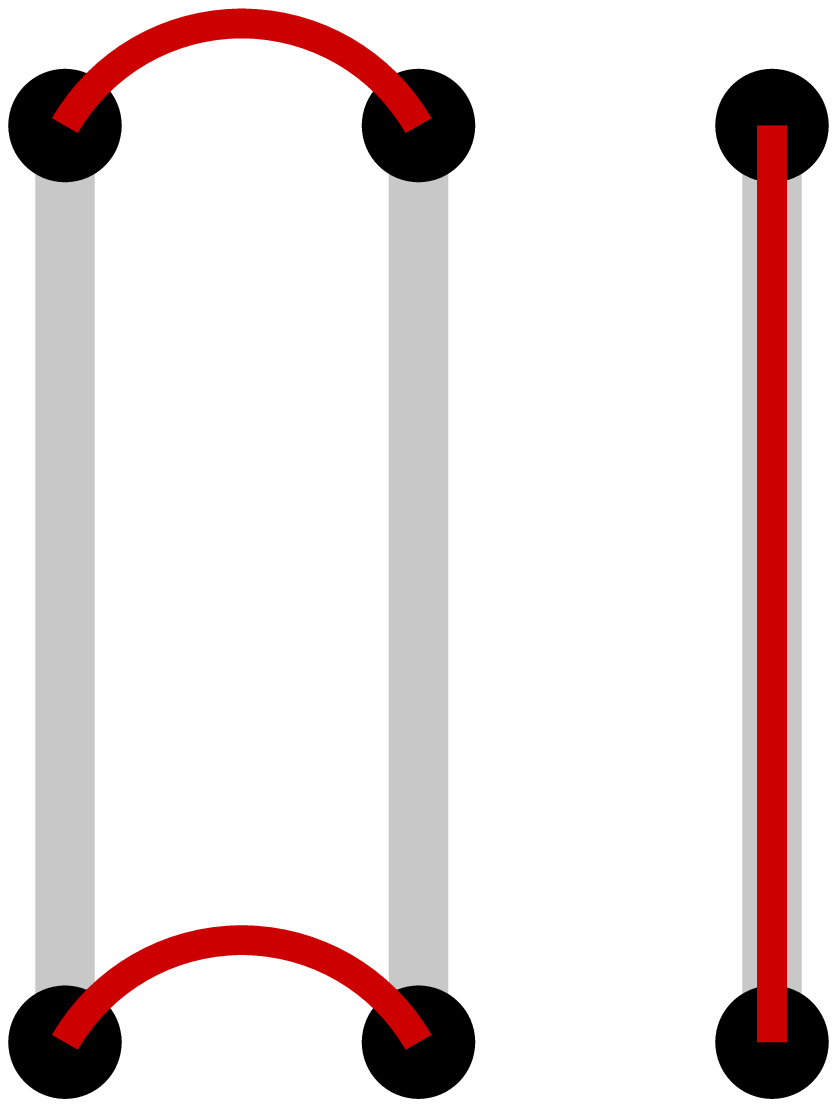}}
\caption{\label{fig:b60}}
\end{figure}

\begin{figure}[htb]
\centering
\subfigure[$n/16$\label{fig:b61.1}]{\includegraphics[scale=0.20]{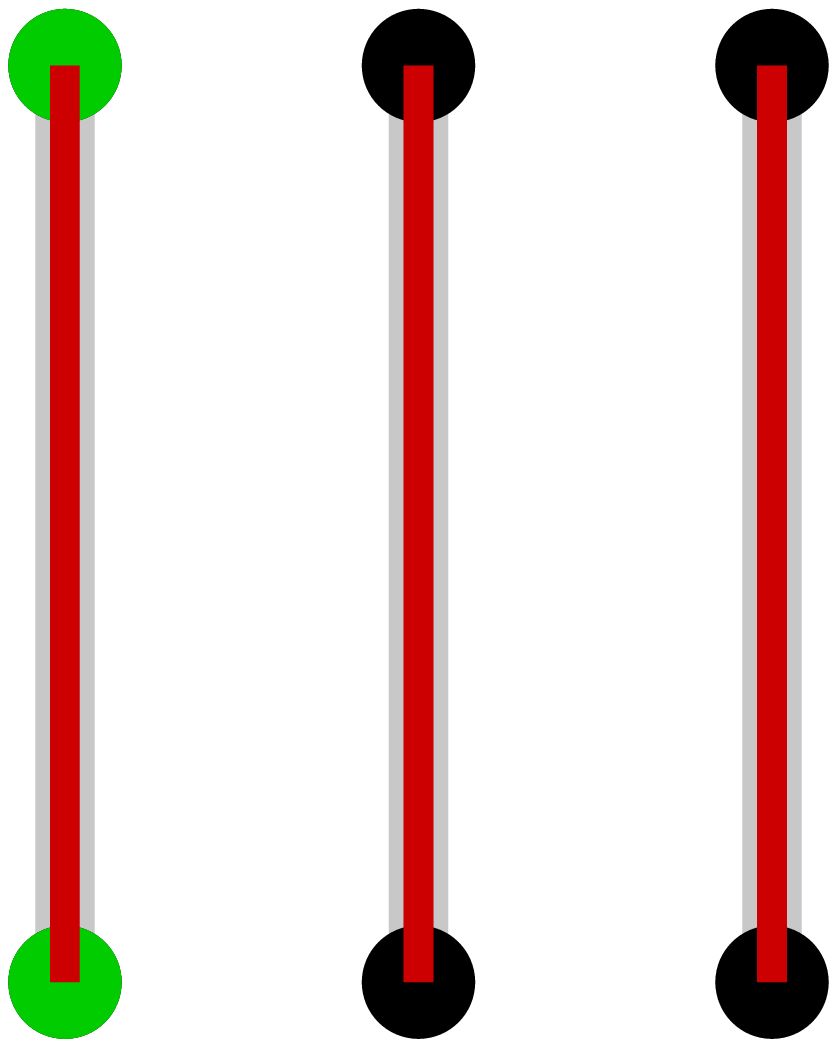}}
\quad\quad\quad 
\subfigure[$n/8$\label{fig:b61.2}]{\includegraphics[scale=0.20]{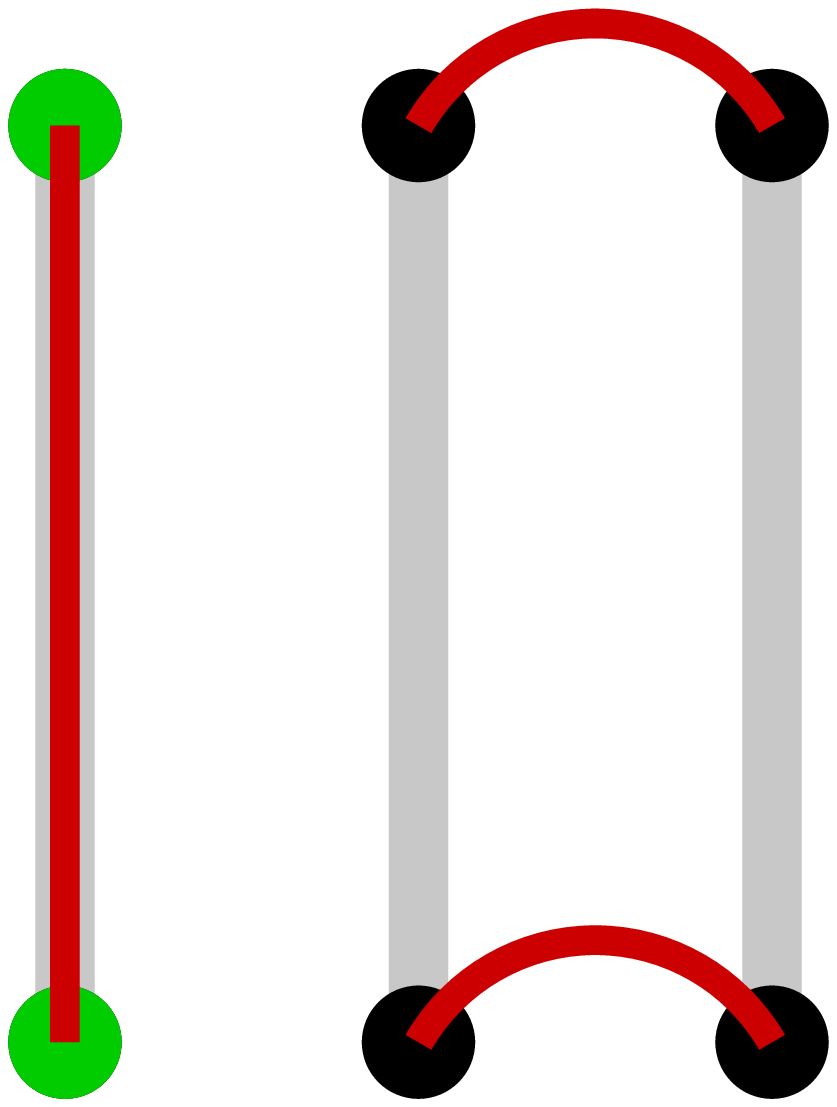}}
\quad\quad\quad 
\subfigure[$n/2$\,$\star$\label{fig:b61.3}]{\includegraphics[scale=0.20]{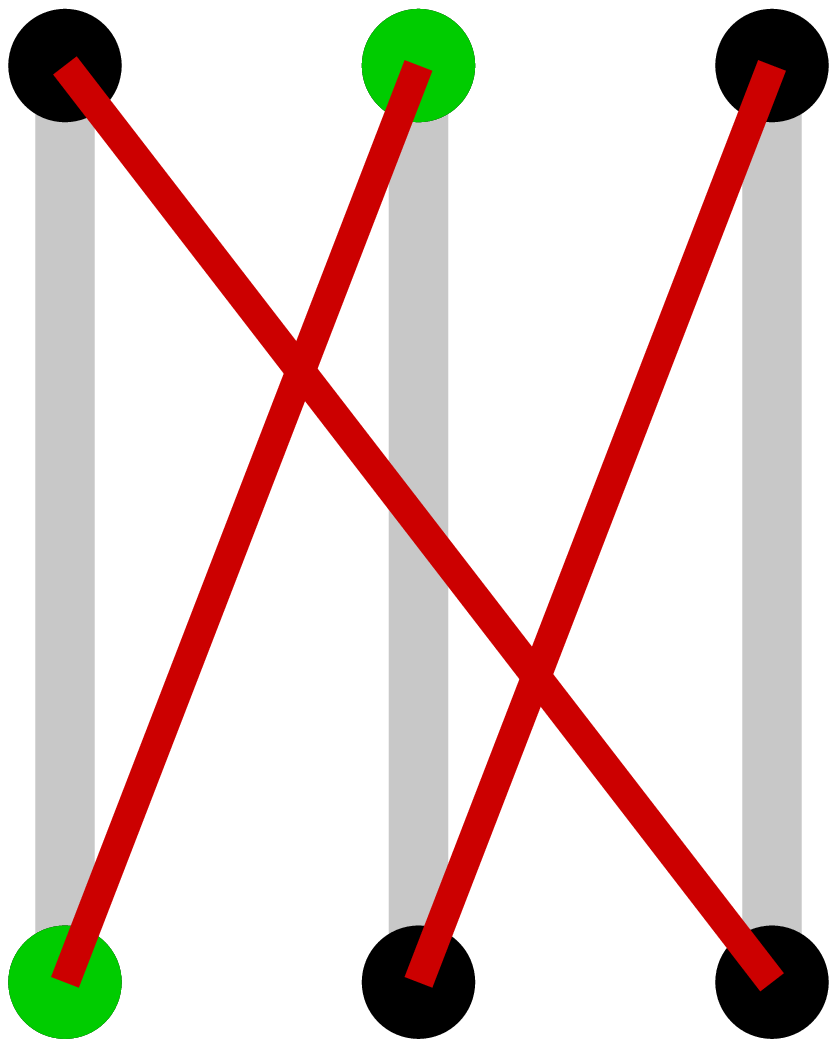}}
\quad\quad\quad 
\subfigure[$n/4$\label{fig:b61.4}]{\includegraphics[scale=0.20]{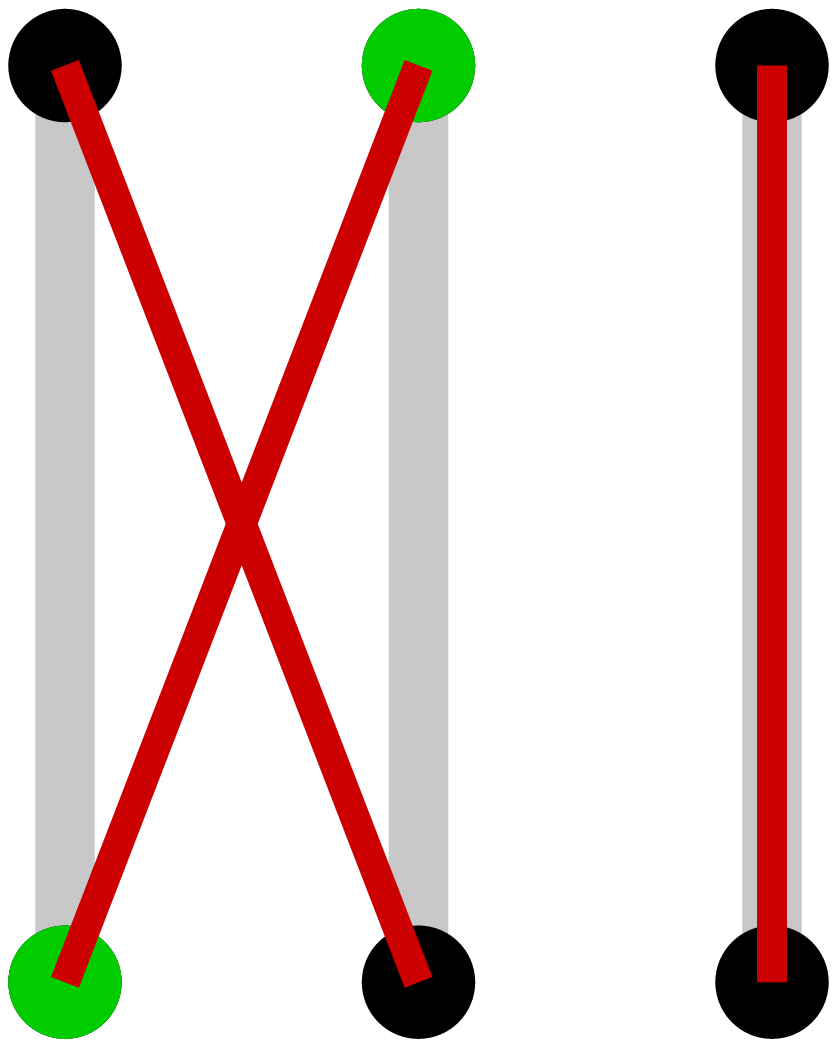}}
\caption{\label{fig:b61}}
\end{figure}

\begin{figure}[htb]
\centering
\subfigure[$n^{2}/16$\label{fig:b62.1}]{\includegraphics[scale=0.20]{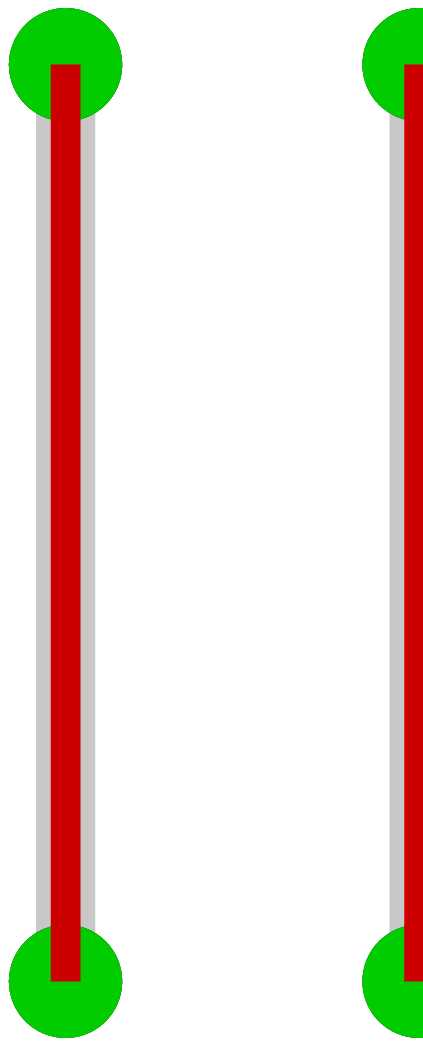}}
\quad\quad\quad 
\subfigure[$n/8$\label{fig:b62.2}]{\includegraphics[scale=0.20]{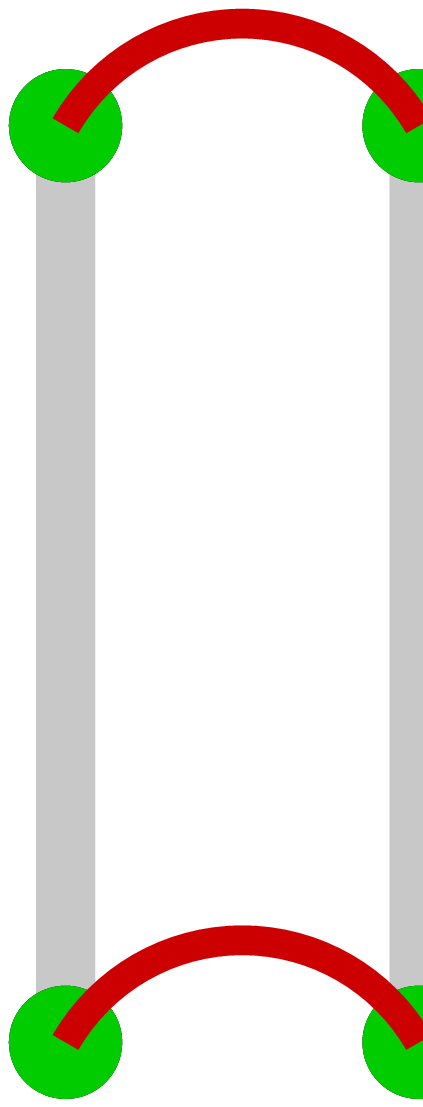}}
\quad\quad\quad 
\subfigure[$n/2$\,$\star$\label{fig:b62.3}]{\includegraphics[scale=0.20]{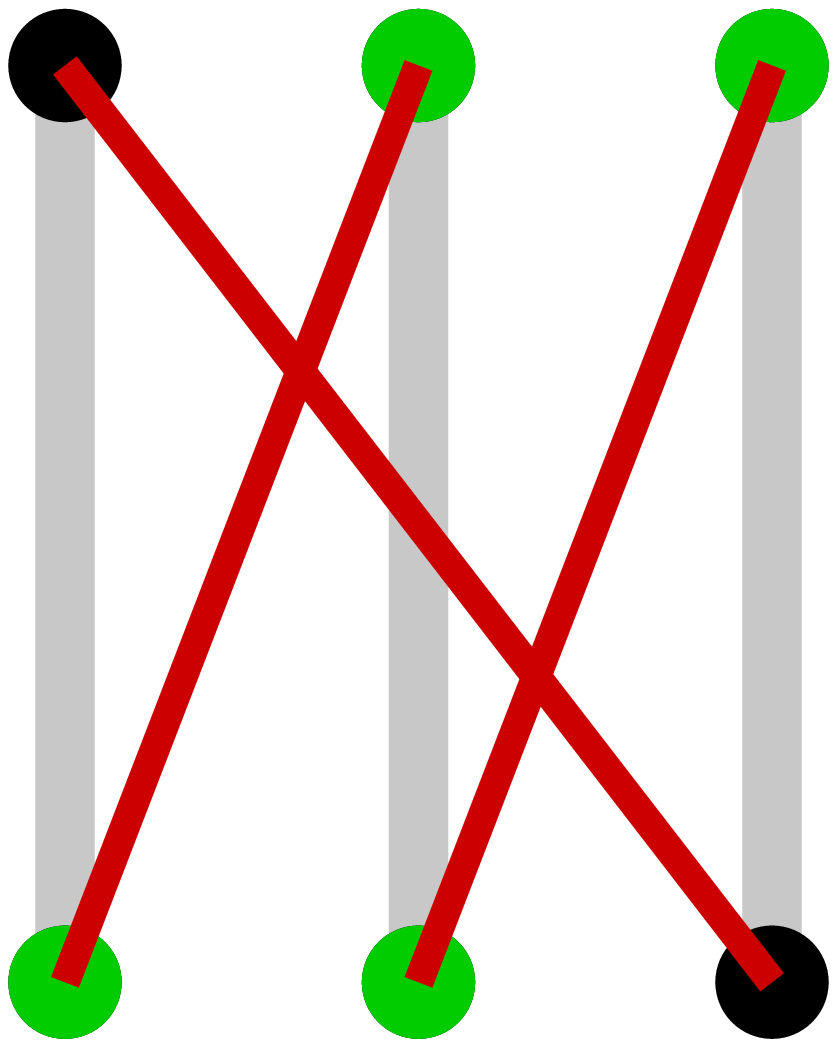}}
\quad\quad\quad 
\subfigure[$n^{2}/4$\label{fig:b62.4}]{\includegraphics[scale=0.20]{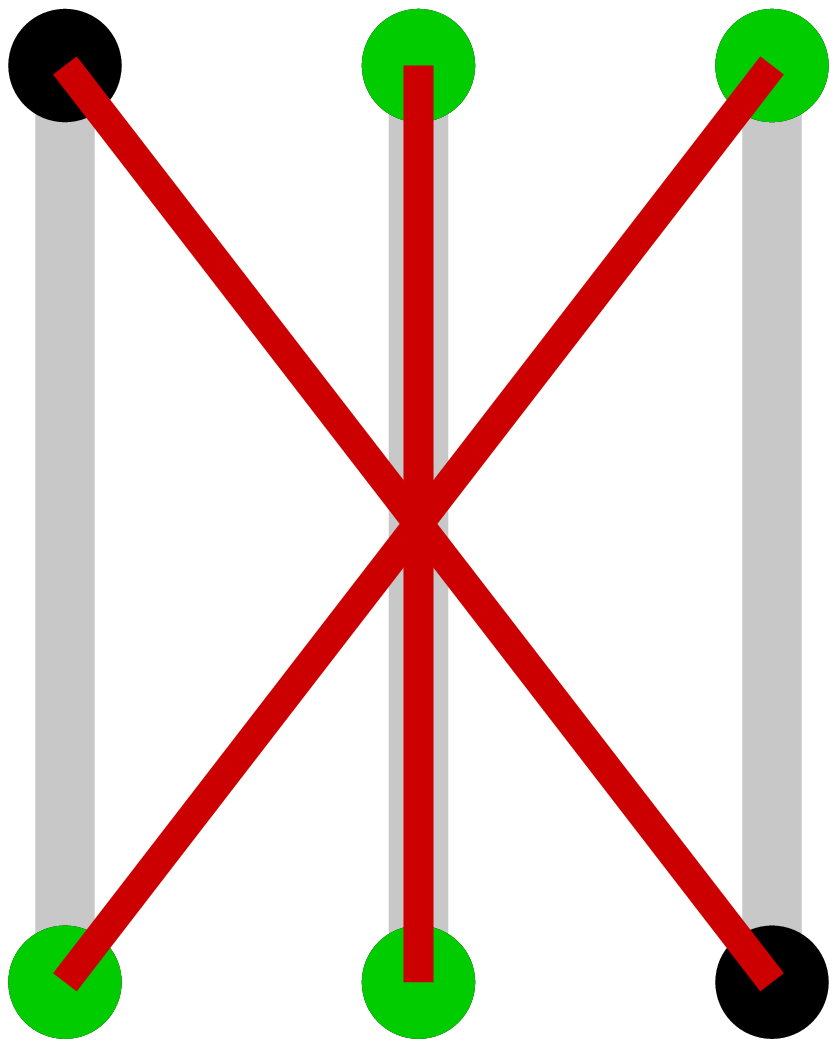}}
\caption{\label{fig:b62}}
\end{figure}

\begin{figure}[htb]
\centering
\subfigure[$n^{3}/48$\label{fig:b63.1}]{\includegraphics[scale=0.20]{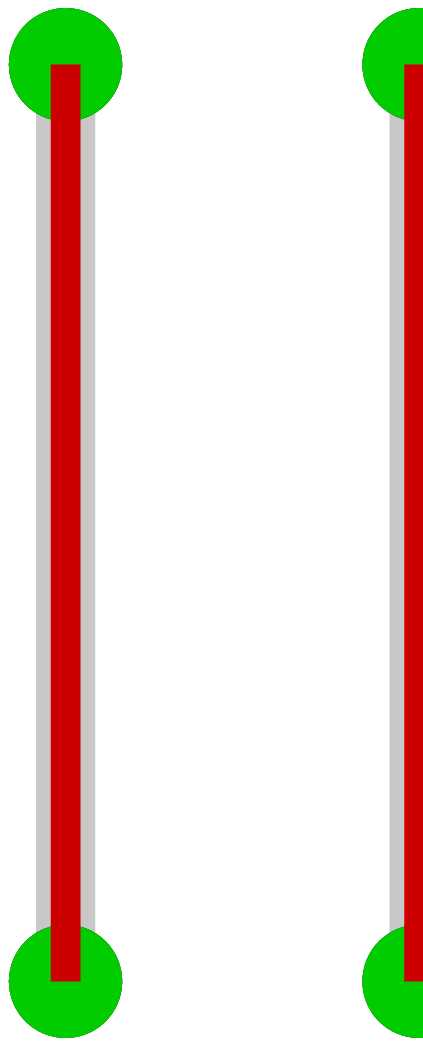}}
\quad\quad\quad 
\subfigure[$n/6$\,$\star$\label{fig:b63.2}]{\includegraphics[scale=0.20]{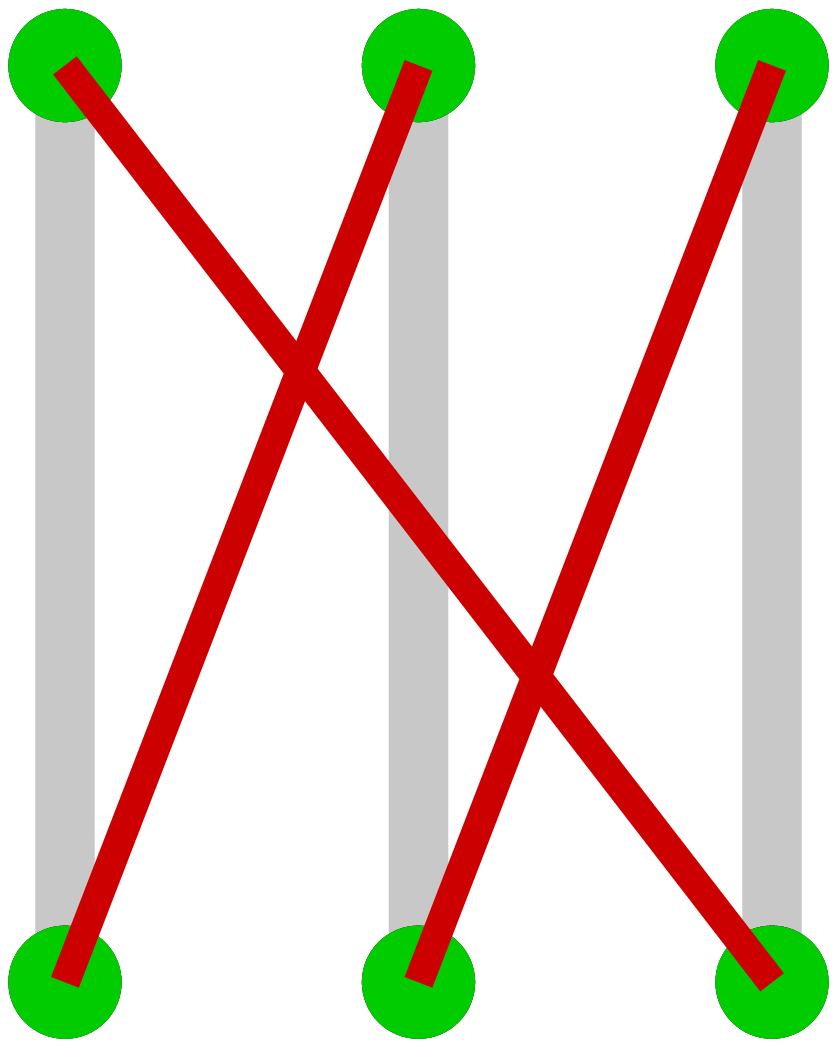}}
\quad\quad\quad 
\subfigure[$n^{2}/8$\label{fig:b63.3}]{\includegraphics[scale=0.20]{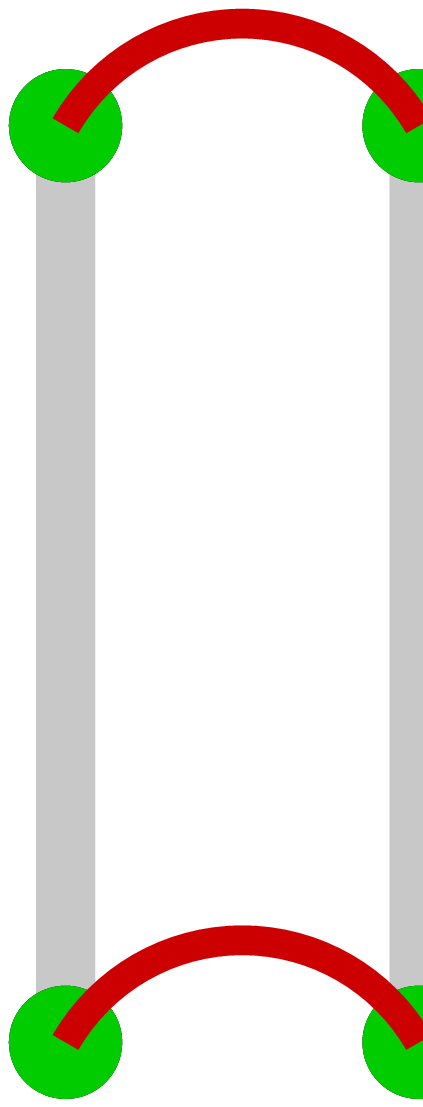}}
\caption{\label{fig:b63}}
\end{figure}

\section{The perturbation series}\label{s:sf:perturb} We conclude by
computing the pertubation series for the spin factor.  We use notation
from \S \ref{s:aa:perturb}.  In particular, let $t$ be an
indeterminate and let $g_{3}, g_{4},\dotsc$ be formal deformation
parameters, packaged together into a vector $\g =
(g_{3},g_{4},\dotsc)$.  Let $\m = (m_{3},m_{4},\dotsc )\in
\prod_{k\geq 3}\ZZ_{\geq 0}$ be a vector of multiplicities; we assume
$m_{k} = 0$ for all sufficiently large $k$.  We write $\g^{\m}$
for the monomial $\prod g_{k}^{m_{k}}$ and set  $N (\m) = \sum k m_{k}$.
We will also need to consider $\Xi_{k}$ and the barbell group for odd $k$.  Thus
for any $k$ define 
\[
m = m (k) = \begin{cases}
k/2& \text{$k$ even,}\\
(k-1)/2 & \text{$k$ odd.}
\end{cases}
\]
Let $M (k) = 2^{m}m!$ and let $B_{m}$ be the barbell group of order $M
(k)$.  For $k$ even we think of $B_{m}$ acting on a set of $k$ fixed
barbells as described in \S \ref{ss:barbellgroup}.  For $k$ odd we
let $\Xi_{k}$ be the set $\{1,\dotsc ,k \}$ with $(k-1)/2$ fixed barbells as in
Figure \ref{fig:barbellgroupodd}, and let $B_{m}$ act by fixing the
isolated point.  Eventually when building barbell graphs using copies
of $\Xi_{k}$, we will assume that the vertices
are always colored with $|S_{gr}|$ \emph{even}.

\begin{figure}[htb]
\psfrag{0}{$1$}
\psfrag{1}{$2$}
\psfrag{2}{$3$}
\psfrag{3}{$4$}
\psfrag{4}{$5$}
\psfrag{5}{$6$}
\psfrag{6}{$7$}
\psfrag{km1}{$k-1$}
\psfrag{k}{$k$}
\psfrag{d}{$\dotsb $}
\begin{center}
\includegraphics[scale=0.25]{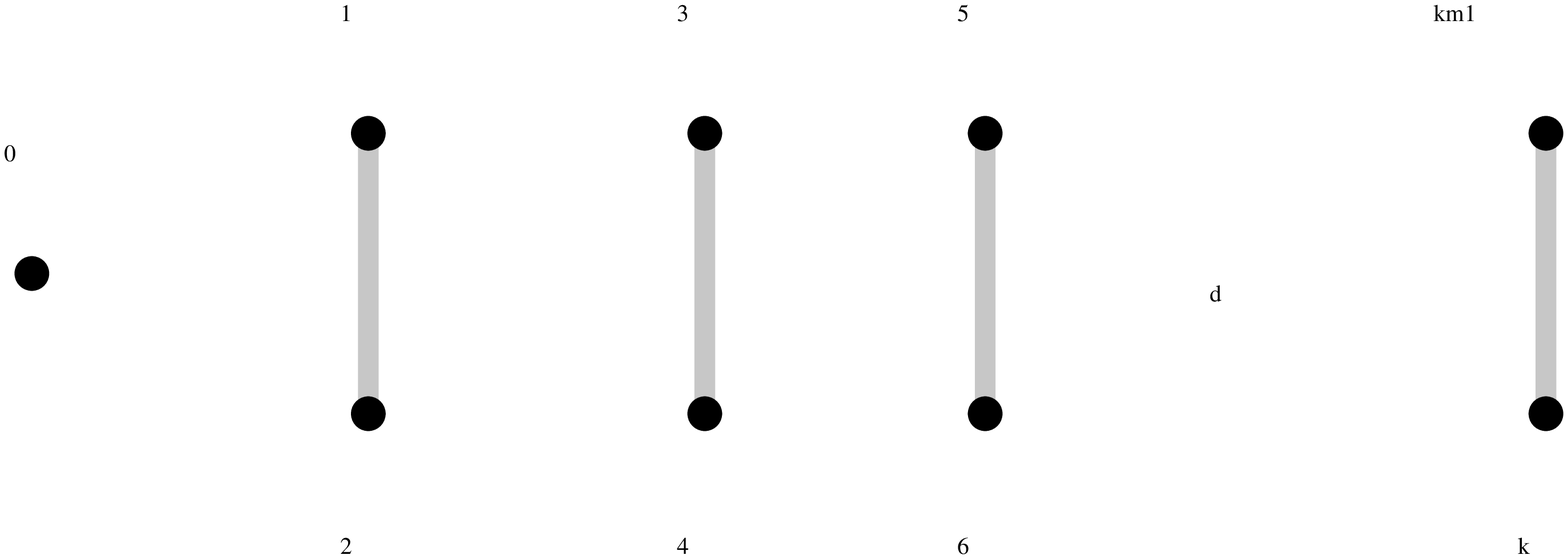}
\end{center}
\caption{\label{fig:barbellgroupodd}}
\end{figure}

Define the pertubation series
\begin{equation}\label{eq:sf:perturb2}
F (\g, t, \x )  = \exp \bigl(\sum_{k \geq 3}(g_{k}\Tr \x ^{k})t^{k}/M (k)\bigr).
\end{equation}
Our goal is to compute the coefficient of $t^{N}$ in
\[
\ipL{F (\g,t,X)}  \in \QQ [g_{3},g_{4},\dotsc ]\fps{t}.
\]
\eqref{eq:perturb2} in terms of barbell graphs. As in \S
\ref{s:aa:perturb}, this coefficient is a homogeneous polynomial of
degree $N$ in the monomials $\g^{\m}$, where $N = N (\m)$, so we
compute the coefficient of each $\g^{\m}t^N$.  Let $\Xi_{\m}$ be the
union $\coprod \Xi_{k}$, where we take $m_{k}$ copies of $\Xi_{k}$.
Let $G (\m)$ be the collection of barbell graphs built from
$\Xi_{\m}$, where in every block of barbells $\Xi_{k}$ the number of
green vertices is even.

\begin{thm}\label{th:sf:perturb}
In \eqref{eq:sf:perturb2}, the coefficient of $\g^{\m}$ in the
coefficient of $t^{N (\m )}$ is
\[
\sum_{\Gamma \in G (\m)} \frac{n^{N (\Gamma)} }{|\Aut \Gamma |},
\]
where $\Aut \Gamma$ consists of automorphisms induced by acting by the
barbell groups in the blocks and permutations of the blocks with the
same number of vertices.
\end{thm}

\begin{proof}
Just like Theorem \ref{th:perturb}, the proof follows from the
exponential formula for generating functions together with Theorem
\ref{th:sf:connected}.  The only subtleties are that (i) one must have
an even number of green vertices in each subcollection of barbells,
and (ii) the connected components need not be closed $1$-manifolds,
but can be $1$-manifolds with boundary.  The first follows from the
trace formula in Proposition \ref{prop:tracepolys}, and the second
follows since we now allow odd numbers of vertices in each block of
barbells.
\end{proof}

\begin{ex}\label{ex:sf:3.3}
We give an example to show how to apply Theorem \ref{th:sf:perturb}.
We compute the coefficient of $g_{3}^{2}$, which corresponds to
$\ipL{(\Tr \x ^{3})^{2}}/8$.  Thus $\m = (2,0,\dotsc)$ and we are
building barbell graphs from two blocks, each a copy of $\Xi_{3}$.
The denominator $8$ comes from the product $2\cdot 2 \cdot 2 = M
(3)\cdot M (3) \cdot 2!$.  In particular, we have two actions of the
barbell group $B_{1}$ in each $\Xi_{3}$, and we also have the
involution that exchanges the blocks.  The barbell graphs are shown in
Figures \ref{fig:t330}--\ref{fig:t334}, along with the quantities
$n^{N (\Gamma)}/|\Aut \Gamma |$.  The result is
\[
\frac{\ipL{(\Tr \x ^{3})^{2}}}{2\cdot 2\cdot 2!}=\frac{9}{8}\*n^2
 + \frac{7}{2}\*n
 + \frac{15}{8}.
\]
\end{ex}

\begin{figure}[htb]
\psfrag{8}{$1/8$}
\psfrag{4}{$1/4$}
\psfrag{2}{$1/2$}
\begin{center}
\includegraphics[scale=0.25]{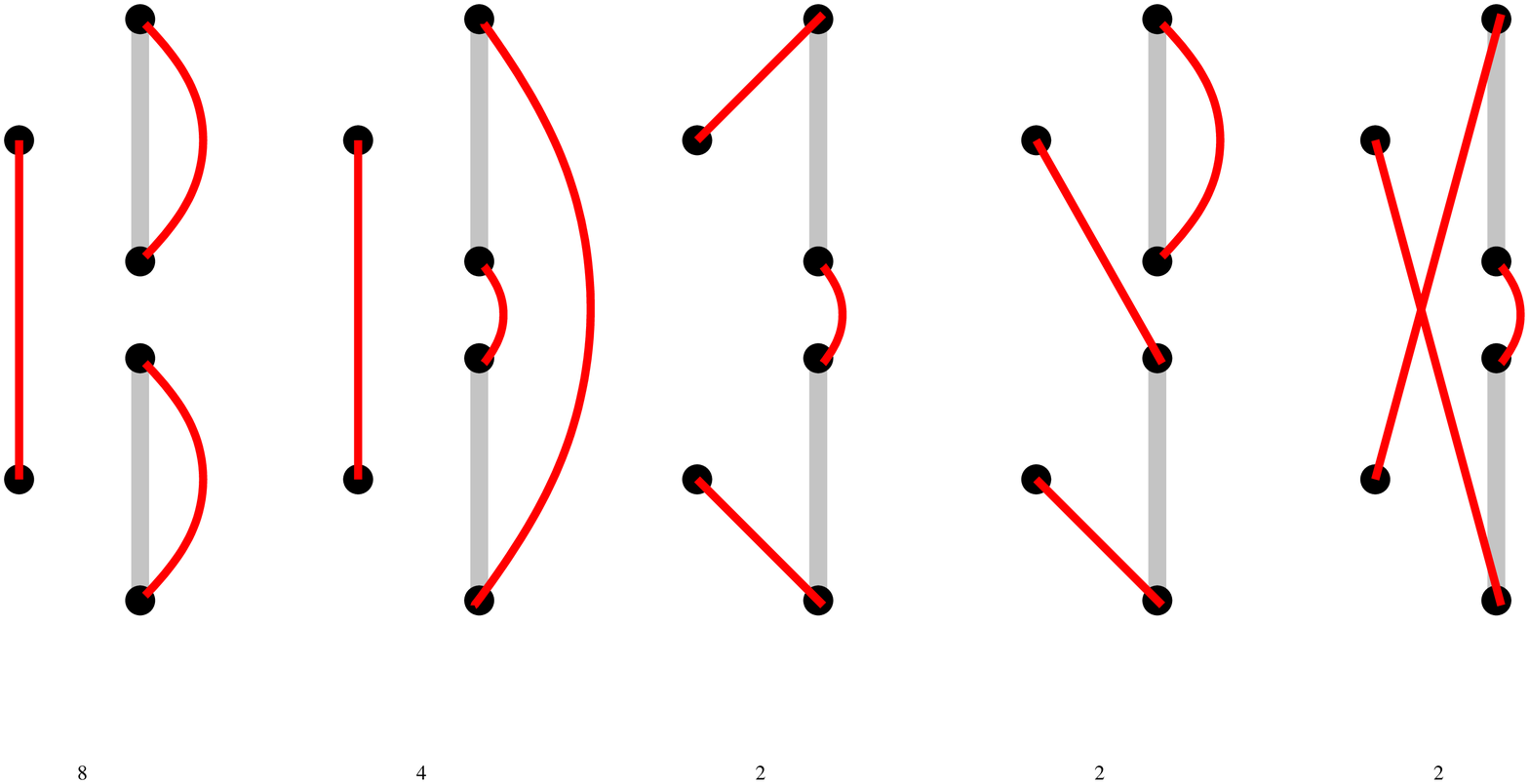}
\end{center}
\caption{\label{fig:t330}}
\end{figure}

\begin{figure}[htb]
\psfrag{n4}{$n/4$}
\psfrag{n2}{$n/2$}
\psfrag{n1}{$n$}
\begin{center}
\includegraphics[scale=0.25]{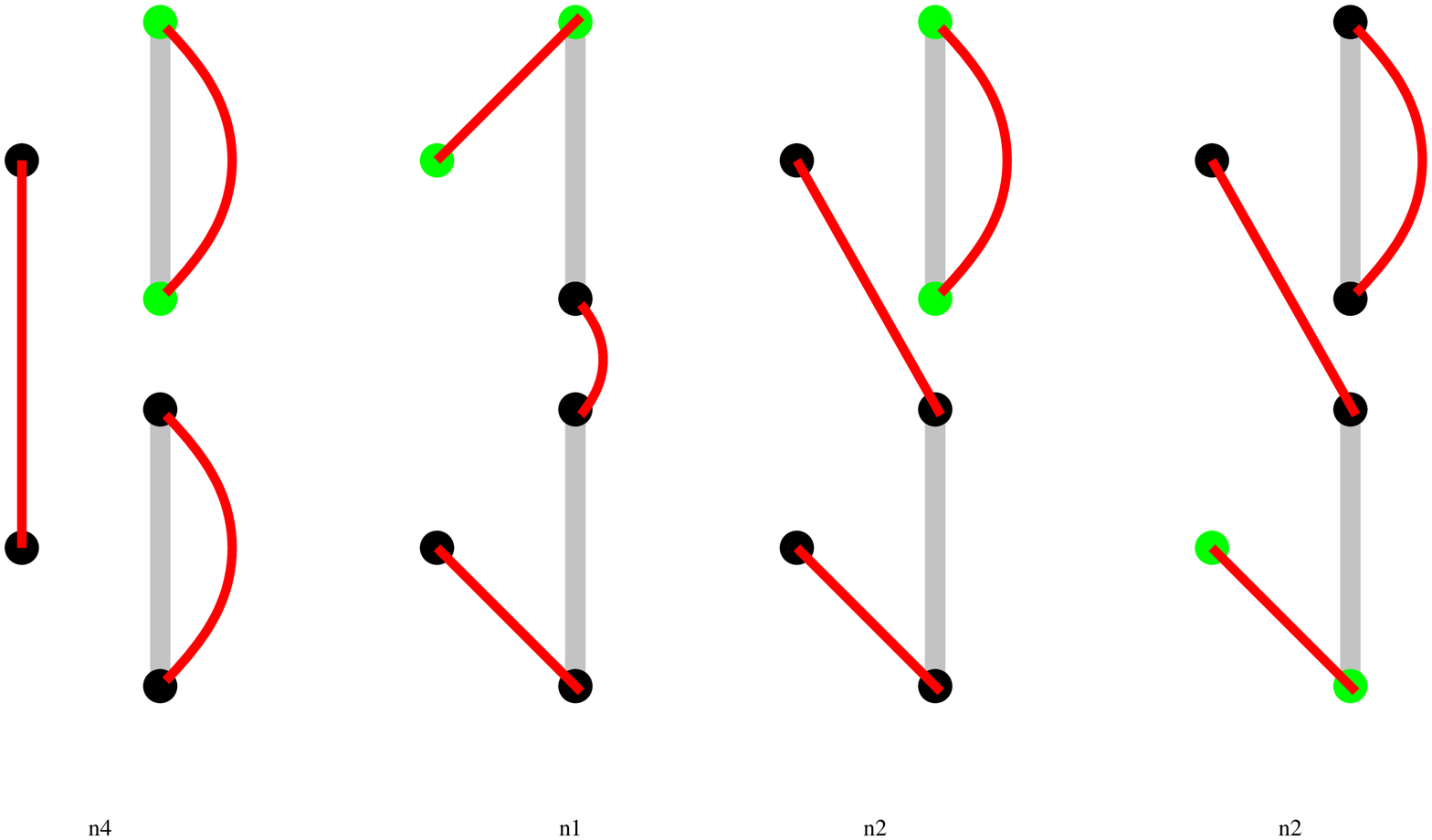}
\end{center}
\caption{\label{fig:t332}}
\end{figure}

\begin{figure}[htb]
\psfrag{n4}{$n/4$}
\psfrag{n2}{$n/2$}
\psfrag{n22}{$n^{2}/2$}
\psfrag{n28}{$n^{2}/8$}
\begin{center}
\includegraphics[scale=0.25]{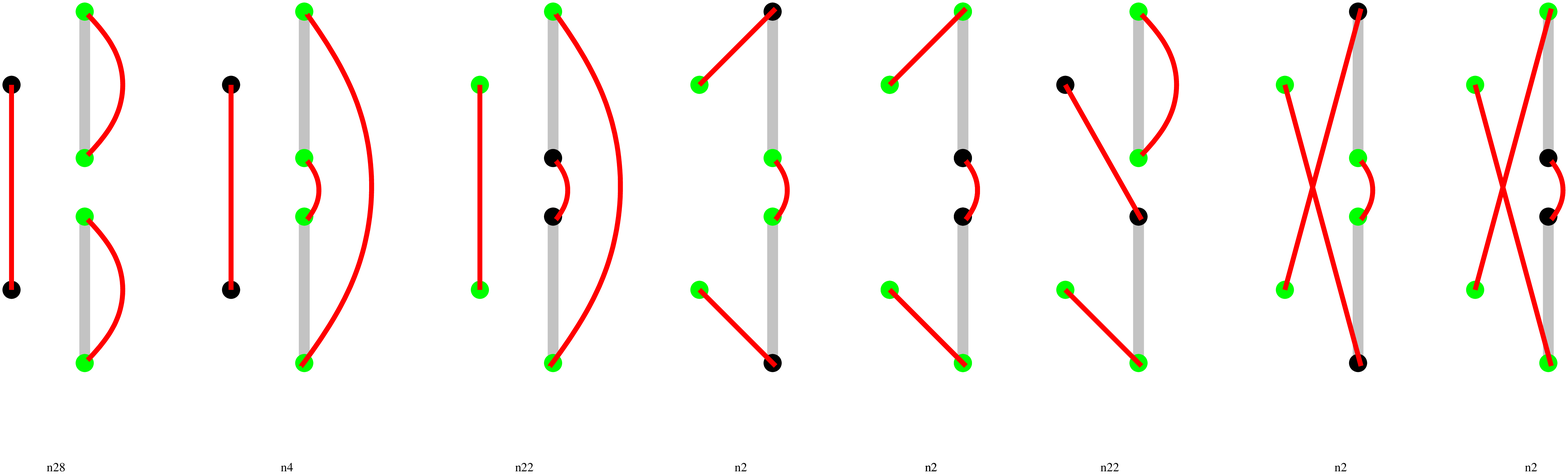}
\end{center}
\caption{\label{fig:t334}}
\end{figure}

\bibliographystyle{amsplain_initials_eprint_doi_url}
\bibliography{exotic}

\providecommand{\bysame}{\leavevmode\hbox to3em{\hrulefill}\thinspace}
\providecommand{\MR}{\relax\ifhmode\unskip\space\fi MR }
\providecommand{\MRhref}[2]{%
  \href{http://www.ams.org/mathscinet-getitem?mr=#1}{#2}
}
\providecommand{\href}[2]{#2}
\begin{thebibliography}{10}

\bibitem{amrt}
A.~Ash, D.~Mumford, M.~Rapoport, and Y.-S. Tai, \emph{Smooth compactifications
  of locally symmetric varieties}, second ed., Cambridge Mathematical Library,
  Cambridge University Press, Cambridge, 2010, With the collaboration of Peter
  Scholze.

\bibitem{etingof}
P.~Etingof, \emph{Mathematical ideas and notions of quantum field theory},
  available from \texttt{www-math.mit.edu/$\sim$etingof/}, 2002.

\bibitem{harer.zagier}
J.~Harer and D.~Zagier, \emph{The {E}uler characteristic of the moduli space of
  curves}, Invent. Math. \textbf{85} (1986), no.~3, 457--485.

\bibitem{mirror}
K.~Hori, S.~Katz, A.~Klemm, R.~Pandharipande, R.~Thomas, C.~Vafa, R.~Vakil, and
  E.~Zaslow, \emph{Mirror symmetry}, Clay Mathematics Monographs, vol.~1,
  American Mathematical Society, Providence, RI; Clay Mathematics Institute,
  Cambridge, MA, 2003, With a preface by Vafa.

\bibitem{jnw}
P.~Jordan, J.~von Neumann, and E.~Wigner, \emph{On an algebraic generalization
  of the quantum mechanical formalism}, Ann. of Math. (2) \textbf{35} (1934),
  no.~1, 29--64.

\bibitem{koecher}
M.~Koecher, \emph{The {M}innesota notes on {J}ordan algebras and their
  applications}, Lecture Notes in Mathematics, vol. 1710, Springer-Verlag,
  Berlin, 1999, Edited, annotated and with a preface by Aloys Krieg and
  Sebastian Walcher.

\bibitem{lz}
S.~K. Lando and A.~K. Zvonkin, \emph{Graphs on surfaces and their
  applications}, Encyclopaedia of Mathematical Sciences, vol. 141,
  Springer-Verlag, Berlin, 2004, With an appendix by Don B. Zagier,
  Low-Dimensional Topology, II.

\bibitem{mccrimmon}
K.~McCrimmon, \emph{A taste of {J}ordan algebras}, Universitext,
  Springer-Verlag, New York, 2004.

\bibitem{mulase.waldron}
M.~Mulase and A.~Waldron, \emph{Duality of orthogonal and symplectic matrix
  integrals and quaternionic {F}eynman graphs}, Comm. Math. Phys. \textbf{240}
  (2003), no.~3, 553--586.

\bibitem{oeis}
N.~J.~A. Sloane, \emph{{Online Encyclopedia of Integer Sequences}}, available
  at \texttt{oeis.org}.

\end{thebibliography}
\end{document}